\documentclass[11pt,reqno]{amsart}

\usepackage[T1]{fontenc}
\usepackage[utf8]{inputenc}
\usepackage[margin=1in]{geometry}
\usepackage{amsmath,amssymb,amsthm,mathtools}
\usepackage{microtype}
\usepackage{xcolor}
\usepackage{enumitem}
\usepackage{tikz}
\usetikzlibrary{arrows.meta,calc,positioning,shapes.geometric}
\usepackage[colorlinks=true,linkcolor=teal!60!black,citecolor=teal!60!black,urlcolor=teal!60!black]{hyperref}

\numberwithin{equation}{section}

\newtheorem{theorem}{Theorem}[section]
\newtheorem{proposition}[theorem]{Proposition}
\newtheorem{lemma}[theorem]{Lemma}

\theoremstyle{definition}
\newtheorem{remark}[theorem]{Remark}
\newtheorem{example}[theorem]{Example}

\newcommand{\RR}{\mathbb{R}}
\newcommand{\ZZ}{\mathbb{Z}}

\newcommand{\C}{\mathbb{C}}

\newcommand{\Ord}{\operatorname{ord}}
\newcommand{\ex}{\operatorname{ex}}
\newcommand{\bits}{\{0,1\}}
\newcommand{\E}{\mathbb E}
\newcommand{\Pbb}{\mathbb P}
\newcommand{\e}[1]{e^{2\pi i #1}}
\newcommand{\ch}{\operatorname{ch}}

\newcommand{\eps}{\varepsilon}

\newcommand{\Arg}{\operatorname{Arg}}

\title{Short proofs in combinatorics, probability and number theory II}
\author{
Boris Alexeev, Moe Putterman, Mehtaab Sawhney, Mark Sellke, and Gregory Valiant\\
OpenAI
}

\email{\{balexeev,mputt,msawhney,msellke,valiant\}@openai.com}
\date{\today}

\begin{document}

\begin{abstract}
We give a quintet of proofs resulting from questions posed by Erd\H{o}s. These questions concern ordinary lines in planar point sets, sequences with uniformly small exponential sums, $K_4$-free $4$-critical graphs with few chords in any cycle, a counterexample to a ``fewnomial'' version of the Erd\H{o}s--Tur\'{a}n discrepancy bound, and a finiteness theorem for integers $n$ such that $n-a k^2$ is prime for all $k\leq \sqrt{n/a}$ coprime to $n$ (for fixed $a\in\mathbb Z_+$). Each proof is due to an internal model at OpenAI.
\end{abstract}

\maketitle

\section{Introduction}\label{sec:intro}

This note collects solutions to five different problems of Erd\H{o}s in a single manuscript. The presentation is inspired by a series of papers by Alon and by Conlon, Fox and Sudakov \cite{Alon03,Alon08,Alon16,alon2020problems,CFS14,CFS16}: each section states one problem, summarizes the relevant prior literature, and then gives the proof.

Section~\ref{sec:960} concerns a question from \cite{Er84} about planar point sets with no $k$ collinear points and no $r$-point set whose pairwise connecting lines are all ordinary (i.e. contain no third point from the same set).  Erd\H{o}s hoped that the largest possible number of ordinary lines under such a forbidden-clique constraint might be $o(n^2)$, or even $O(n)$.  We disprove all non-trivial cases of this conjecture by constructing sets with no four collinear points, a triangle-free (in fact bipartite) ordinary-line graph, and $\Omega(n^2)$ ordinary lines; our construction takes place within a large cyclic subgroup of a real elliptic curve. We remark that earlier work of F\"uredi--Pal\'asti and Escudero \cite{FuPa84,Esc16} gave a collection of $d$ points with no four on a line but no triple of ordinary lines which form a triangle; the key improvement therefore is finding a construction with quadratically many ordinary lines. The use of cubic curves in connection with ordinary lines also appears in the classical orchard-problem; we refer the reader to the paper of Green--Tao \cite{GT13} for further information. 

Section~\ref{sec:987} answers a question of Erd\H{o}s from \cite{Erd64,Erd65}, later recorded by Hayman \cite{Hay74}, about whether one can have $\widetilde A_k:=\limsup_{N\to\infty}|\sum_{0\leq n<N} e^{2\pi i k x_n}|=o(k)$ for all $k\in\mathbb N$ and some fixed real sequence $(x_i)_{i\geq 0}$.  Clunie \cite{Cl67} proved the lower bound $\widetilde A_k\gg k^{1/2}$ must hold for infinitely many $k$ and gave a deterministic dyadic sequence with $A_k:=\sup_{N\geq 1}|\sum_{0\leq n<N} e^{2\pi i k x_n}|\le k$. Our \emph{randomized} dyadic construction in Section~\ref{sec:987} satisfies $A_k\ll \sqrt{k\log(k)}$, nearly matching Clunie's lower bound.

Section~\ref{sec:1091} disproves a conjecture of Erd\H{o}s \cite{Er76c} asking whether a chromatic number $4$ graph such that every ``small subgraph'' has chromatic number at most $3$ contains a cycle with many chords. Voss proved that every $K_4$-free $4$-chromatic graph has an odd cycle with at least two chords, building on Larson's work \cite{Vo82,La79}.
 Section~\ref{sec:1091} constructs explicit arbitrarily large $K_4$-free $4$-chromatic graphs for which all proper subgraphs are $2$-degenerate, yet every cycle has at most ten chords.

Section~\ref{sec:990} disproves a natural sparse analogue of the Erd\H{o}s--Tur\'an theorem (see e.g. \cite{Sound}) with the degree $d$ of a polynomial $f$ replaced by the number of nonzero coefficients $\nu(f)$ in the discrepancy bound for arguments of zeros; this answers a question raised by Erd\H{o}s \cite{Erd64}. We remark that a result of Hayman \cite{Hayman} implies that the discrepancy in roots is always bounded by $\le \nu(f)-1$. The fewnomial family in Section~\ref{sec:990} has $\nu(f)=N+2$, bounded coefficient growth parameter $M(f)$, and a positive real root of multiplicity $N+1$; thus no bound of order $\sqrt{\nu(f)\log M(f)}$ can hold uniformly.

Finally, Section~\ref{sec:nt} proves that for each fixed integer $a\ge 1$, only finitely many integers $n$ have the property that $n-a k^2$ is prime for every $k$ with $a k^2<n$ and $(k,n)=1$. For $a=1$ this is Erd\H{o}s's Problem~1141 (see \cite{Erd99} and \cite[Problem~1141]{BloWeb}). 
In the case $a=2$, such a finiteness result was previously known in the easier setting where the condition $(k,n)=1$ is not enforced (already disproving \cite[Problem~1140]{BloWeb}). Our argument is a short deduction from a result of Pollack \cite[Theorem~1.3]{Pollack17} on small prime quadratic residues.

\subsection*{Comment on the use of AI}

The proofs in this manuscript are due to an internal model at OpenAI. 
In each case, after verifying the internal model solution, we asked ChatGPT-5.4 Pro five independent times to solve the same problem.
The only successful attempts were \href{https://chatgpt.com/share/69d2896e-20f4-8333-89a1-f9b86eb7f95f}{this shared ChatGPT transcript} on \cite[Problem~960]{BloWeb}
(the subject of Section~\ref{sec:960}) and all five attempts including \href{https://chatgpt.com/share/69d2ec5b-3460-8326-b7e6-9ff6e61f8e25}{this shared ChatGPT transcript} on \cite[Problem~1141]{BloWeb} (the subject of Section~\ref{sec:nt}).
For the former problem, ChatGPT's solution follows a similar route to that of the internal model by working inside a cyclic subgroup of a real elliptic curve, but is slightly weaker in that it does not resolve the case $r=3$ (i.e. ChatGPT's construction ensures $K_4$-freeness but not triangle-freeness). 
For the latter, we first asked both models simply to solve Erd{\H{o}}s problem 1141 concerning $n$ such that $n-k^2$ is never prime.
Upon examining the solutions we realized that the method should extend to $n-ak^2$ for any $a$, and posed this as a follow-up query to ChatGPT which readily generalized the proof.

The role of the human authors was simply to digest the proofs and modify the write-ups for clarity and elegance.  The only further (minor) proof-level modification occurs in the argument for \cite[Problem~1091]{BloWeb}, the subject of Section~\ref{sec:1091}.\footnote{The model's original proof provided the same family of example graphs, but deduced color-criticality from a presentation by Haj\'os joins. The proof retained here instead establishes the stronger statement that every proper subgraph is $2$-degenerate. This degeneracy-based route was suggested by the human authors while digesting the model output; because it gives a slightly simpler verification, that version has been retained.}

\subsection*{Correspondence to Erd{\H{o}}s problems website} The \url{erdosproblems.com} website~\cite{BloWeb}, curated by Thomas Bloom, includes the problems from Sections~\ref{sec:960} through~\ref{sec:nt} as Problems 960, 987, 1091, 990, and 1141, respectively.

\section[Problem 960]{Many ordinary lines but no ordinary clique}\label{sec:960}

\subsection{Statement and reformulation}

Fix integers $r,k \geq 2$.
For a finite set $A \subset \RR^2$, write $\Ord(A)$ for the number of lines $\ell$ with $|\ell \cap A|=2$.
For $n \geq r$, define
\[
F_{r,k}(n):=
\max \Ord(A),
\]
where the maximum is taken over all $n$-point sets $A \subset \RR^2$ such that
\[
|\ell \cap A| \leq k-1 \qquad \text{for every line } \ell
\]
and such that $A$ contains no subset $A' \subset A$ with $|A'|=r$ for which every pair of distinct points in $A'$ spans an ordinary line of $A$.
If no such configuration exists, set $F_{r,k}(n)=-1$.

Given $A$, define its \emph{ordinary-line graph} $G_A$ by
\[
V(G_A)=A,
\qquad
\{p,q\} \in E(G_A)
\iff
|\ell_{pq}\cap A|=2,
\]
where $\ell_{pq}$ denotes the line through $p$ and $q$.
Then
\[
e(G_A)=\Ord(A),
\]
and the desired $r$-point subset is exactly a copy of $K_r$ in $G_A$.
Thus $F_{r,k}(n)$ is the maximum number of edges in an ordinary-line graph $G_A$ subject to the geometric constraint ``no $k$ points collinear'' and the graph-theoretic constraint ``$G_A$ is $K_r$-free.''
Figure~\ref{p960:fig:intro-example} depicts a small example.

Erd\H{o}s asked about the asymptotic behavior of this threshold in \cite{Er84}. The closest previous work we are aware of on this conjecture is due to F\"uredi--Pal\'asti and Escudero \cite{FuPa84,Esc16} which gives a set of points with no $4$ collinear but with no triplet of ordinary lines forming a triangle \cite{FuPa84,Esc16}.

We first note that certain cases are immediate. If $k=2$ and $n\geq 2$, then no valid $n$-point set exists at all, so $F_{r,2}(n)=-1$. If $k=3$ then by definition $G_A=K_n$ and so $F_{r,3}(n)=-1$ for all $n\geq r$ (i.e. no valid sets $A$ exist for such $(r,k,n)$). If $r=2$ and $n\geq k\geq 4$, then the Sylvester--Gallai theorem again implies $F_{2,k}(n)=-1$.
Additionally, one has of course $F_{r,k}(n)\leq \binom{n}{2}$, and in fact Tur\'an's theorem gives the improvement
\[
F_{r,k}(n)\leq \ex(n,K_r)\leq \left(1-\frac{1}{r-1}\right)\frac{n^2}{2}
\]
without using the condition on $k$. Erd\H{o}s wrote \cite{Er84} that he hoped the threshold should be $o(n^2)$, and perhaps even $O(n)$. The main result of this section shows that $F_{r,k}(n)\geq \frac{n^2}{12}-O(n)$.
\begin{theorem}
\label{p960:thm:main}
Fix integers $r \geq 3$ and $k \geq 4$ and $n \geq 72$.
Then
    \[
     F_{r,k}(n)
     \geq 
     \frac{n^2}{12}-\frac{10}{3}n
  .
    \]
\end{theorem}

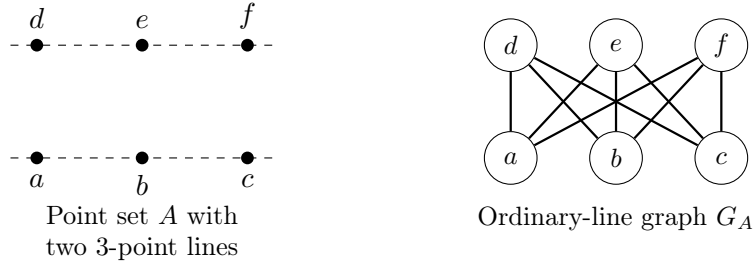
\begin{figure}[t]
\centering
\begin{tikzpicture}[
    pt/.style={circle,fill=black,inner sep=1.7pt},
    gvertex/.style={circle,draw,minimum size=7mm,inner sep=1pt,font=\small},
    gedge/.style={line width=0.85pt}
]
\begin{scope}[xshift=0cm]
  \node[pt,label=below:$a$] (a) at (0,0) {};
  \node[pt,label=below:$b$] (b) at (1.4,0) {};
  \node[pt,label=below:$c$] (c) at (2.8,0) {};
  \node[pt,label=above:$d$] (d) at (0,1.5) {};
  \node[pt,label=above:$e$] (e) at (1.4,1.5) {};
  \node[pt,label=above:$f$] (f) at (2.8,1.5) {};
  \draw[dashed] (-0.35,0) -- (3.15,0);
  \draw[dashed] (-0.35,1.5) -- (3.15,1.5);
  \node[font=\small,align=center] at (1.4,-1) {Point set $A$ with\\ two $3$-point lines};
\end{scope}

\begin{scope}[xshift=6.3cm]
  \node[gvertex] (ga) at (0,0) {$a$};
  \node[gvertex] (gb) at (1.4,0) {$b$};
  \node[gvertex] (gc) at (2.8,0) {$c$};
  \node[gvertex] (gd) at (0,1.5) {$d$};
  \node[gvertex] (ge) at (1.4,1.5) {$e$};
  \node[gvertex] (gf) at (2.8,1.5) {$f$};
  \draw[gedge] (ga) -- (gd);
  \draw[gedge] (ga) -- (ge);
  \draw[gedge] (ga) -- (gf);
  \draw[gedge] (gb) -- (gd);
  \draw[gedge] (gb) -- (ge);
  \draw[gedge] (gb) -- (gf);
  \draw[gedge] (gc) -- (gd);
  \draw[gedge] (gc) -- (ge);
  \draw[gedge] (gc) -- (gf);
  \node[font=\small] at (1.4,-0.8) {Ordinary-line graph $G_A$};
\end{scope}
\end{tikzpicture}
\caption{A simple example of $G_A$.
In the point configuration on the left, the only nonordinary lines are the two lines through
$a,b,c$ and through $d,e,f$.
Accordingly, on the right the graph $G_A$ is the complete bipartite graph $K_{3,3}$.}
\label{p960:fig:intro-example}
\end{figure}

\subsection{Elliptic-curve construction}
We now briefly summarize the construction of the set $A$. The key point is to take a large torsion subgroup $\mathbb{Z}/(7m\mathbb{Z})$ of the real points on an elliptic curve and remove all points in the zero residue class modulo $7$. The ordinary lines then come from pairs $(x,y)$ such that $x+y\equiv 0 \pmod 7$ or $x = -2y$ or $y = -2x$ in $\mathbb{Z}/(7m\mathbb{Z})$. Via direct inspection the ordinary-line graph forms a bipartite graph and this completes the proof when $n$ is divisible by $6$.
When $n$ is not divisible by $6$, a constant number of additional points from the removed residue class are added back to $A$ in an ad-hoc manner (see Subsection~\ref{subsubec:adjust-size}).
We note that while we use a specific elliptic curve below for concreteness, any (non-degenerate) elliptic curve suffices (even in the case of a two-component curve, one just works within the component containing the identity).

\subsubsection{The ambient cubic}

Let $E$ be the projective closure of the affine curve
\[
y^2=x^3-x+1.
\]
Equivalently, in homogeneous coordinates $[X:Y:Z]$ on $\mathbb{P}^2$, the curve $E$ is given by
\[
Y^2Z=X^3-XZ^2+Z^3.
\]
Let $O=(0:1:0)$ be its point at infinity.
Write
\[
E(\RR):=\{[X:Y:Z]\in \mathbb{P}^2(\RR):Y^2Z=X^3-XZ^2+Z^3\}
\]
for the real locus of this projective curve, so $E(\RR)$ consists of the affine real solutions to $y^2=x^3-x+1$ together with the point at infinity $O$.
Additive notation is used for the group law on $E$, with identity $O$.
Only the following standard facts about elliptic curves are needed.
First, the chord--tangent construction turns a smooth plane cubic with a distinguished point $O$ into an abelian group: if a line meets $E$ in three points $x,y,z$, counted with multiplicity, then
\[
x+y+z=O.
\]
See e.g. \cite[Chapter~I, Theorem~3.1 and Proposition~4.10 and Remark~4.11(c)]{Mil21}.
In the affine model $y^2=f(x)$, negation is reflection across the $x$-axis:
\[
-(x,y)=(x,-y).
\]
Second, $\mathbb{P}^2(\RR)$ is compact and $E(\RR)$ is a closed subset of it, hence $E(\RR)$ is compact.
For a real Weierstrass cubic $y^2=f(x)$ with $f$ squarefree, the real locus has one or two connected components according as $f$ has one or three real roots; in the connected case, $E(\RR)$ is isomorphic as a Lie group to a circle.
See \cite[\emph{Introduction to Rational Points on Plane Curves}, \S7, Proposition~7.2]{Hus04}.

\begin{lemma}
\label{p960:lem:connected}
The real locus $E(\RR)$ is connected.
Consequently $E(\RR)$ contains a cyclic subgroup of order $M$ for every integer $M \geq 1$.
\end{lemma}

\begin{proof}
The discriminant of $x^3-x+1$ is
\[
-4(-1)^3-27(1)^2=-23 \neq 0,
\]
so $E$ is smooth.
The cubic polynomial $x^3-x+1$ has exactly one real root, hence the real locus of $E$ is connected.
By the preceding classification of the real locus, it is therefore isomorphic to a circle, hence to $\RR/\ZZ$.
In particular, for every $M \geq 1$ it has a cyclic subgroup of order $M$.
\end{proof}

The standard collinearity criterion on a cubic is also needed:
three points $x,y,z \in E$ are collinear, counted with multiplicity, if and only if
\[
x+y+z=O.
\]

\begin{lemma}
\label{p960:lem:no4}
Every affine line in $\RR^2$ meets $E(\RR)\setminus\{O\}$ in at most three points.
In particular, every finite subset of $E(\RR)\setminus\{O\}$ has no four collinear points.
\end{lemma}

\begin{proof}
A projective line meets the projective cubic $E$ in at most three points, counting multiplicity, by B\'ezout's theorem.
Thus an affine line can contain at most three affine points of $E$.
\end{proof}

\subsubsection{The base set}

Fix an integer $m \geq 1$.
By Lemma~\ref{p960:lem:connected}, choose a cyclic subgroup
\[
C=\langle g\rangle \leq E(\RR),
\qquad
|C|=7m.
\]
Let
\[
H=\langle 7g\rangle \leq C,
\qquad
|H|=m.
\]
For $i \in \ZZ/7\ZZ$, write
\[
C_i=ig+H.
\]
Then
\[
C=C_0 \sqcup C_1 \sqcup \cdots \sqcup C_6,
\qquad
C_0=H.
\]
Define the base configuration
\[
A_0=C\setminus H=C_1\sqcup C_2\sqcup C_3\sqcup C_4\sqcup C_5\sqcup C_6.
\]
Thus $|A_0|=6m$.

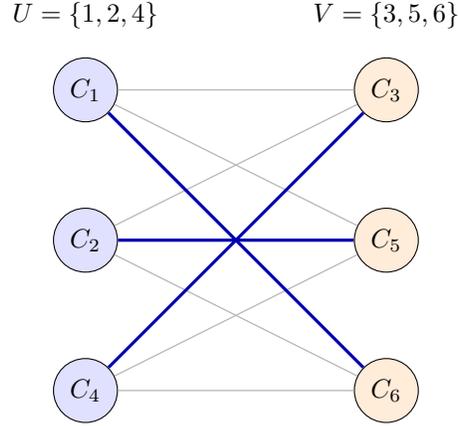
\begin{figure}[t]
\centering
\begin{tikzpicture}[
    coset/.style={circle,draw,minimum size=8.5mm,inner sep=1pt,font=\small},
    ucoset/.style={coset,fill=blue!12},
    vcoset/.style={coset,fill=orange!15}
]
\node[ucoset] (c1) at (0,2) {$C_1$};
\node[ucoset] (c2) at (0,0) {$C_2$};
\node[ucoset] (c4) at (0,-2) {$C_4$};

\node[vcoset] (c3) at (4,2) {$C_3$};
\node[vcoset] (c5) at (4,0) {$C_5$};
\node[vcoset] (c6) at (4,-2) {$C_6$};

\foreach \u in {c1,c2,c4}{
  \foreach \v in {c3,c5,c6}{
    \draw[gray!55,line width=0.5pt] (\u) -- (\v);
  }
}

\draw[blue!70!black,line width=1.2pt] (c1) -- (c6);
\draw[blue!70!black,line width=1.2pt] (c2) -- (c5);
\draw[blue!70!black,line width=1.2pt] (c4) -- (c3);

\node[font=\small] at (0,3) {$U=\{1,2,4\}$};
\node[font=\small] at (4,3) {$V=\{3,5,6\}$};
\end{tikzpicture}
\caption{Coset-level edge pattern in Proposition~\ref{p960:prop:base}. The gray edges are the admissible residue relations
$j\equiv -i$, $j\equiv -2i$, or $j\equiv 3i\pmod 7$, all crossing from $U$ to $V$.
The highlighted opposite pairs $(C_1,C_6)$, $(C_2,C_5)$, and $(C_4,C_3)$ are the three families contributing $m^2$ ordinary edges each.}
\label{p960:fig:coset-pattern}
\end{figure}

\begin{proposition}
\label{p960:prop:base}
The ordinary-line graph $G_{A_0}$ is bipartite, hence triangle-free.
Moreover,
\[
\Ord(A_0)\geq 3m^2.
\]
\end{proposition}

\begin{proof}
Take distinct points $x,y \in A_0$.
Let $z$ be the third point of intersection of the line $\ell_{xy}$ with $E$, counted with multiplicity.
By the cubic group law,
\[
x+y+z=O.
\]
Since $x,y \in C$ and $C$ is a subgroup, one also has $z \in C$.

Because every affine line meets $E$ in at most three points, the line $\ell_{xy}$ contains points of $A_0$ only among $x,y,z$.
Hence $\ell_{xy}$ is ordinary for $A_0$ if and only if either $z \in H$, or $z=x$, or $z=y$.
The three cases are:
\begin{align*}
z \in H
&\iff
x+y \in H,\\
z=x
&\iff
y=-2x,\\
z=y
&\iff
x=-2y.
\end{align*}

Now suppose $x \in C_i$ and $y \in C_j$.
If $\{x,y\}$ is an edge of $G_{A_0}$, then one of the following must hold:
\[
j \equiv -i \pmod 7,
\qquad
j \equiv -2i \pmod 7,
\qquad
j \equiv 3i \pmod 7.
\]
Consider the partition of the six nonzero residues modulo $7$ into
\[
U=\{1,2,4\},
\qquad
V=\{3,5,6\}.
\]
The multipliers $-1$, $-2$, and $3$ all send $U$ onto $V$.
At the coset level, this gives the bipartite pattern shown in Figure~\ref{p960:fig:coset-pattern}.
Therefore every edge of $G_{A_0}$ joins a point of
\[
X=C_1\sqcup C_2\sqcup C_4
\]
to a point of
\[
Y=C_3\sqcup C_5\sqcup C_6.
\]
Therefore $G_{A_0}$ is bipartite, and in particular triangle-free.

To count edges, consider the three opposite coset pairs
\[
(C_1,C_6),\qquad (C_2,C_5),\qquad (C_4,C_3).
\]
If $x \in C_i$ and $y \in C_{-i}$, then $x+y \in H$, so $z=-x-y$ lies in $H$ and $\ell_{xy}$ is ordinary for $A_0$.
Each of the three opposite coset pairs contributes exactly $m^2$ ordinary edges.
Therefore
$\Ord(A_0)\geq 3m^2$.
\end{proof}

\subsubsection{Adjusting the size}
\label{subsubec:adjust-size}

To obtain every large $n$, add a small set inside $H$.
Write
\[
n=6m+s,
\qquad
0 \leq s \leq 5.
\]
For the argument below, assume $m \geq 12$, equivalently $n \geq 72$.
Fix a generator $h$ of $H$ and define
\begin{align*}
T_0&=\varnothing,\\
T_1&=\{h\},\\
T_2&=\{h,2h\},\\
T_3&=\{h,2h,-3h\},\\
T_4&=\{h,2h,-3h,3h\},\\
T_5&=\{h,2h,-3h,3h,-4h\}.
\end{align*}
Because $m \geq 12$, all listed points are distinct and nonzero, so $|T_s|=s$.
Set
\[
A=A_0 \cup T_s.
\]
Then $|A|=6m+s=n$.

\begin{figure}[t]
\centering
\begin{tikzpicture}[
    vertex/.style={circle,draw,minimum size=7mm,inner sep=1pt,font=\scriptsize},
    edge/.style={line width=0.9pt}
]
\begin{scope}[xshift=0cm]
  \node[vertex] (a1) at (0,0.95) {$h$};
  \node[vertex] (a2) at (-1.0,-0.65) {$2h$};
  \node[vertex] (a3) at (1.0,-0.65) {$-3h$};
  \node[font=\small] at (0,-1.65) {$s=3$ (edgeless)};
\end{scope}

\begin{scope}[xshift=4.8cm]
  \node[vertex] (b0) at (0,0) {$3h$};
  \node[vertex] (b1) at (-1.05,0.95) {$h$};
  \node[vertex] (b2) at (-1.05,-0.95) {$2h$};
  \node[vertex] (b3) at (1.05,0) {$-3h$};
  \draw[edge] (b0) -- (b1);
  \draw[edge] (b0) -- (b2);
  \draw[edge] (b0) -- (b3);
  \node[font=\small] at (0,-1.65) {$s=4$ (star)};
\end{scope}

\begin{scope}[xshift=9.6cm]
  \node[vertex] (c0) at (-1.45,1.35) {$h$};
  \node[vertex] (c1) at (-0.75,0) {$2h$};
  \node[vertex] (c2) at (0,0.95) {$3h$};
  \node[vertex] (c3) at (0.75,0) {$-3h$};
  \node[vertex] (c4) at (0,-0.95) {$-4h$};
  \draw[edge] (c1) -- (c2);
  \draw[edge] (c2) -- (c3);
  \draw[edge] (c3) -- (c4);
  \draw[edge] (c4) -- (c1);
  \node[font=\small] at (0,-1.65) {$s=5$ ($4$-cycle + isolated $h$)};
\end{scope}
\end{tikzpicture}
\caption{The induced graphs on $T_s$ in the only nontrivial cases $s=3,4,5$.
These are exactly the configurations used in Proposition~\ref{p960:prop:final-construction} to rule out triangles in $G_A[T_s]$.}
\label{p960:fig:Ts-graphs}
\end{figure}
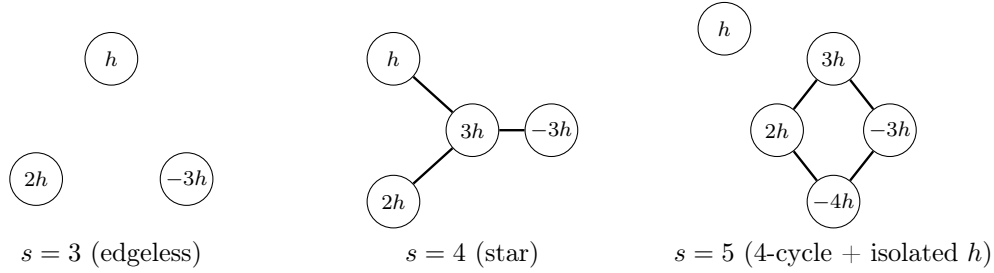

\begin{proposition}
\label{p960:prop:final-construction}
For every $n \geq 72$, the set $A$ defined above satisfies the following.

\begin{enumerate}
    \item No four points of $A$ are collinear.

    \item The ordinary-line graph $G_A$ is bipartite, in particular triangle-free.

    \item
    $\Ord(A)\geq 3m^2-3ms$.
\end{enumerate}
\end{proposition}

\begin{proof}
Since $A \subset E(\RR)\setminus\{O\}$, Lemma~\ref{p960:lem:no4} gives (1).
Next we analyze ordinary edges.

\medskip
\noindent
\emph{No edges join $A_0$ to $T_s$.}
Take $t \in T_s \subset H$ and $x \in A_0$.
If $x \in C_i$ with $i \neq 0$, then the third point on the line through $t$ and $x$ is
\[
z=-t-x \in C_{-i}\subset A_0.
\]
This point is distinct from both $t$ and $x$, so the line through $t$ and $x$ is not ordinary.
Thus there are no edges between $A_0$ and $T_s$.

\medskip
\noindent
\emph{The graph induced by $T_s$ is bipartite.}
Since $T_s \subset H$, for any distinct $u,v \in T_s$ the third point of the line through $u$ and $v$ also lies in $H$.
So ordinariness inside $T_s$ can be checked purely inside the cyclic group $H$.
A direct computation gives the graphs shown in Figure~\ref{p960:fig:Ts-graphs}:
\begin{itemize}
    \item for $s=0,1,2$, the graph on $T_s$ has clique number at most $2$;
    \item for $s=3$, there are no edges at all, because $h+2h+(-3h)=O$;
    \item for $s=4$, the only edges are
    \[
    \{h,3h\},\qquad \{2h,3h\},\qquad \{3h,-3h\},
    \]
    which form a star;
    \item for $s=5$, the only edges are
    \[
    \{2h,3h\},\qquad \{2h,-4h\},\qquad \{3h,-3h\},\qquad \{-4h,-3h\},
    \]
    which form a $4$-cycle.
\end{itemize}
Hence $G_A[T_s]$ is bipartite in every case.

\medskip
\noindent
\emph{The graph induced by $A_0$ stays bipartite.}
If a pair $\{x,y\}\subset A_0$ is nonordinary in $A_0$, then its third point already lies in $A_0$, and so it remains nonordinary after adding $T_s$.
Thus $G_A[A_0]$ is a subgraph of the bipartite graph $G_{A_0}$ from Proposition~\ref{p960:prop:base}.
Combining the preceding paragraphs proves (2).

For (3), count only the ordinary edges coming from the three opposite coset pairs
\[
(C_1,C_6),\qquad (C_2,C_5),\qquad (C_4,C_3).
\]
By Proposition~\ref{p960:prop:base}, these give $3m^2$ ordinary edges in $A_0$.
When a point $t \in T_s \subset H$ is added, such an edge $\{x,y\}$ disappears exactly when its third point is $t$, equivalently when
\[
x+y=-t.
\]
Fix one opposite pair $(C_i,C_{-i})$ and one $t \in H$.
For each $x \in C_i$, there is a unique $y=-t-x \in C_{-i}$.
Hence exactly $m$ edges from that opposite pair are destroyed by $t$.
There are three opposite pairs, so each added point destroys exactly $3m$ of the previously counted edges.
Since $|T_s|=s$, at least
\[
3m^2-3ms
\]
of those edges survive in $A$.
Therefore $\Ord(A)\geq 3m^2-3ms$.
\end{proof}

\subsection{Proof of the main theorem}

\begin{proof}[Proof of Theorem~\ref{p960:thm:main}]
Fix $n \geq 72$, write $n=6m+s$ with $0 \leq s \leq 5$, and construct the set $A$ from Proposition~\ref{p960:prop:final-construction}.
Part (1) of that proposition shows that $A$ has no four collinear points, hence certainly no $k$ collinear points.
Part (2) shows that $G_A$ is triangle-free, so in particular $G_A$ contains no $K_r$.
Therefore $A$ is a valid configuration, so
\[
F_{r,k}(n)\geq \Ord(A)\geq 3m^2-3ms.
\]
Since $n=6m+s$,
\[
3m^2-3ms
=
\frac{n^2}{12}-\frac{2s}{3}n+\frac{7s^2}{12}
\geq
\frac{n^2}{12}-\frac{10}{3}n,
\]
because $0 \leq s \leq 5$.
Hence
\[
F_{r,k}(n)\geq \frac{n^2}{12}-\frac{10}{3}n.
\qedhere
\]
\end{proof}

\section[Problem 987]{A randomized sequence with uniformly small exponential sums}\label{sec:987}

This problem concerns real sequences $x_0,x_1,\dots$ such that every initial segment $(x_i)_{0\leq i<N}$ has complex exponential sums of small norm.
The construction is a prefix-randomized version of the binary van der Corput sequence. 
On each dyadic block $[P2^r,(P+1)2^r)$ of length $2^r$, the right-most $r$ scrambled digits of $i$, which run through all binary words exactly once, are used to generate the first $r$ binary digits of $x_i\in [0,1)$. 
Meanwhile the remaining tail digits of $x_i$ are conditionally independent.
This gives a sum of independent centered random variables on each dyadic block, enabling multi-scale estimates after carefully identifying the right prefix set to union bound over.

\subsection{Introduction}

Given a sequence $(x_n)_{n \ge 0}$ in $\mathbb{R}/\mathbb{Z}$ and an integer $k \ge 1$, set
\[
S_N(k) := \sum_{n=0}^{N-1} \e{k x_n}
\qquad (N \ge 1).
\]
We show there exists a sequence $(x_n)_{n \ge 0}$ in $\mathbb{R}/\mathbb{Z}$ such that
\[
A_k:=\sup_{N \ge 1} |S_N(k)| \ll \sqrt{k \log(2k)}
\qquad (k \ge 1).
\]
In particular this answers a question of Erd\H{o}s \cite{Erd64,Erd65}, recorded as Problem 7.21 in Hayman's problem list \cite{Hay74} and listed on the Erd\H{o}s problems website as Problem~987 \cite{BloWeb}, which asks whether it is possible that
\[
\widetilde A_k:=\limsup_{N\to\infty}|S_N(k)|=o(k).
\]
Of course $A_k\geq \widetilde A_k$; the present and previous work gives upper bounds for $A_k$ and lower bounds for $\widetilde A_k$.
Erd\H{o}s observed in \cite{Erd64} that $\widetilde A_k$ always diverges,
and later gave a very easy proof that $\widetilde A_k \gg \log k$ for infinitely many $k$ \cite{Erd65}. Clunie \cite{Cl67} proved the much stronger universal lower bound
\[
\widetilde A_k \gg k^{1/2}
\]
for infinitely many $k$, and also gave an explicit sequence with $A_k\le k$ for all $k$. In particular our improved upper bound is sharp up to the logarithmic factor.

\begin{theorem}\label{p987:thm:main}
There exists a sequence $(x_n)_{n \ge 0}$ in $\mathbb{R}/\mathbb{Z}$ such that
\[
A_k := \sup_{N \ge 1} |S_N(k)| \ll \sqrt{k \log(2k)}
\qquad (k \ge 1).
\]
\end{theorem}



\subsection*{Example: \texorpdfstring{$N=2026$}{N=2026}}

Before giving the construction, it is helpful to see exactly how the final estimate will be organized.
Our sums are indexed from $0$, so we naturally decompose the interval $[0,N)=\{0,1,\dots,N-1\}$ into dyadic blocks.

For the concrete value
\[
2026 = 11111101010_2
=
2^{10}+2^9+2^8+2^7+2^6+2^5+2^3+2,
\]
the binary digits tell us exactly which dyadic block lengths appear.  Namely,
\[
[0,2026)
=
[0,2^{10})
\cup
[2^{10},2^{10}+2^9)
\cup
[2^{10}+2^9,2^{10}+2^9+2^8)
\cup \cdots \cup
[2024,2026),
\]
that is,
\[
\begin{array}{rcl}
[0,2026)
&=&
[0,1024)
\cup [1024,1536)
\cup [1536,1792)
\cup [1792,1920) \\[2mm]
&& {}\cup [1920,1984)
\cup [1984,2016)
\cup [2016,2024)
\cup [2024,2026).
\end{array}
\]

Each chunk above has the form
\[
[P2^r,(P+1)2^r)
\]
for some integers $P,r$.  Later we will attach to such a chunk a block sum $B_{P,r}(k)$, and Proposition~\ref{p987:prop:block-identity} will show that
\[
\sum_{n=P2^r}^{(P+1)2^r-1} \e{k x_n} = B_{P,r}(k),
\]
while Proposition~\ref{p987:prop:block-bound} will show that if we write
\[
b=b(k):=\lceil \log_2(2k)\rceil,
\]
then every chunk of length $2^r$ satisfies
\[
\left|\sum_{n=P2^r}^{(P+1)2^r-1} \e{k x_n}\right|
=
|B_{P,r}(k)|
\ll
\sqrt{r+b}\,\min\{2^{r/2},2^{b-r/2}\}.
\]
These terms switch behavior at the frequency scale $r=b$: for short blocks relative to $k$ one sees the factor $2^{r/2}$, while for long blocks one sees the decaying factor $2^{b-r/2}$.
Hence the main term is of order $\sqrt{r2^r}\asymp \sqrt{k\log k}$, with geometry decay in both directions.

The construction is easiest to understand on dyadic blocks.  Inside a block of length $2^r$, the first $r$ scrambled binary digits of the points run through a permutation of all $r$-bit words, while the remaining tail pieces become independent and uniformly distributed after conditioning on the randomness from shorter prefixes.  This turns each dyadic block sum into a sum of independent centered random variables.  Once those block sums are controlled uniformly, arbitrary partial sums are handled by decomposing the initial interval $[0,N)$ into dyadic blocks.
We note that this can be seen as a randomized improvement of Clunie's linear upper bound in \cite{Cl67}, which sets $x_1=1, x_2=-1$ and then uses the deterministic recursion 
\[
x_{a+2^r}
=
x_a e^{\pi i/2^r}
\]
when $2^r\geq a\geq 1$ and $x_a$ has already been defined.

\subsection{The binary scrambling}

Fix independent Bernoulli random variables
\[
\eta_u \in \bits,
\qquad
\Pbb(\eta_u = 0) = \Pbb(\eta_u = 1) = \tfrac12,
\]
indexed by all finite binary words $u$ (including the empty word $\varnothing$).

Write a nonnegative integer $n$ in binary as
\[
n = \sum_{i=0}^\infty d_i 2^i,
\qquad
d_i \in \bits.
\]
Read the binary digits from low significance to high significance.  At stage $i$ the prefix $d_0 \cdots d_{i-1}$ has already been seen, and the random bit $\eta_{d_0 \cdots d_{i-1}}$ decides whether to keep or flip $d_i$.  Thus define
\begin{equation}\label{p987:eq:xn}
x_n := \sum_{i=0}^\infty \frac{d_i \oplus \eta_{d_0 \cdots d_{i-1}}}{2^{i+1}},
\end{equation}
where $\oplus$ denotes addition mod $2$.

\begin{remark}
If all the $\eta_u$ were equal to $0$, then \eqref{p987:eq:xn} would be the binary van der Corput sequence.
\end{remark}

For a binary word $w = w_0 \cdots w_{r-1} \in \bits^r$, define
\begin{equation}\label{p987:eq:jrw}
j_r(w) := \sum_{i=0}^{r-1} \bigl(w_i \oplus \eta_{w_0 \cdots w_{i-1}}\bigr) 2^{r-1-i}.
\end{equation}
Thus $j_r(w)/2^r$ is the binary fraction whose first $r$ digits are the scrambled versions of the bits of $w$.

If $P = \sum_{\ell=0}^\infty p_\ell 2^\ell$ is a nonnegative integer, define the scrambled tail after the prefix $w$ by
\begin{equation}\label{p987:eq:tail}
T_{w,P} := \sum_{\ell=0}^\infty \frac{p_\ell \oplus \eta_{w p_0 \cdots p_{\ell-1}}}{2^{\ell+1}}.
\end{equation}
For $k \ge 1$ and $r \ge 0$, set
\begin{equation}\label{p987:eq:block-sum}
B_{P,r}(k) := \sum_{w \in \bits^r} \e{\frac{k}{2^r}\bigl(j_r(w) + T_{w,P}\bigr)}.
\end{equation}

\begin{lemma}\label{p987:lem:jr-bijection}
For each $r \ge 0$, the map $w \mapsto j_r(w)$ is a bijection from $\bits^r$ onto $\{0,1,\dots,2^r-1\}$.
\end{lemma}

\begin{proof}
The binary digits of $j_r(w)$ are precisely
\[
y_i := w_i \oplus \eta_{w_0 \cdots w_{i-1}}
\qquad (0 \le i < r),
\]
written from most to least significant.  Knowing the output bits $y_0,\dots,y_{r-1}$, one reconstructs $w_0$ from $y_0$ and $\eta_\varnothing$, then $w_1$ from $y_1$ and $\eta_{w_0}$, etc.  Thus $w$ is uniquely determined by $j_r(w)$.
\end{proof}

The point of \eqref{p987:eq:block-sum} is that it is exactly the exponential sum over a dyadic block of indices.

\begin{proposition}\label{p987:prop:block-identity}
Let $P \ge 0$ have binary expansion
\[
P = \sum_{\ell=0}^\infty p_\ell 2^\ell,
\qquad
p_\ell \in \bits,
\]
and let $0 \le m < 2^r$ have binary digits $m = \sum_{i=0}^{r-1} w_i 2^i$.  Then
\[
x_{P2^r + m} = \frac{j_r(w) + T_{w,P}}{2^r}.
\]
Consequently
\begin{equation}\label{p987:eq:dyadic-block-identity}
\sum_{n=P2^r}^{(P+1)2^r-1} \e{k x_n} = B_{P,r}(k).
\end{equation}
\end{proposition}

\begin{proof}
The low $r$ binary digits of $P2^r + m$ are $w_0,\dots,w_{r-1}$, and the higher binary digits are $p_0,p_1,\dots$.  Therefore \eqref{p987:eq:xn} splits into the first $r$ scrambled digits and the remaining tail:
\[
x_{P2^r+m}
=
	\sum_{i=0}^{r-1}\frac{w_i \oplus \eta_{w_0 \cdots w_{i-1}}}{2^{i+1}}
	+ \sum_{\ell=0}^\infty
	\frac{p_\ell \oplus \eta_{w p_0 \cdots p_{\ell-1}}}{2^{r+\ell+1}}
	=
	\frac{j_r(w) + T_{w,P}}{2^r}.
	\]
Summing over the $2^r$ choices of $m$ is exactly \eqref{p987:eq:dyadic-block-identity}.
\end{proof}

\subsection{A uniform estimate for dyadic blocks}

The heart of the matter is a bound that is uniform in the block length, the block location, and the frequency scale.

\begin{proposition}[Uniform dyadic block estimate]\label{p987:prop:block-bound}
There is an absolute constant $A > 0$ and a deterministic choice of the bits $(\eta_u)$ such that the following holds.  For every $r \ge 0$, every integer $P \ge 0$, and every integer $k\ge1$, if
\[
b:=b(k)=\lceil \log_2(2k)\rceil,
\]
one has
\[
|B_{P,r}(k)|
\le
A \sqrt{r+b}\, \min\{2^{r/2}, 2^{b-r/2}\}.
\]
\end{proposition}

\noindent Before proving Proposition~\ref{p987:prop:block-bound}, isolate the only place where the truncation length $h$ enters.  Fix a block location $P$, and let
\[
Q:=P \bmod 2^h,
\qquad 0\le Q<2^h.
\]
Then $P$ and $Q$ have the same lowest $h$ binary digits.  Since the block sum depends on $P$ only through those binary digits and the later tail they generate, this implies
\[
|B_{P,r}(k)-B_{Q,r}(k)|\ll 2^b2^{-h}
\]
when $b=b(k)$.  Thus to obtain accuracy $t$, it is enough to choose $h$ so that $2^b2^{-h}\lesssim t$ and then check only the finitely many residues $Q \in \{0,\dots,2^h-1\}$.  In the proof of Proposition~\ref{p987:prop:block-bound} this choice is made separately for each dyadic block scale $r$ and dyadic frequency scale $b$: after fixing $(r,b)$ we set
\[
t_{r,b}:=A\sqrt{r+b}\,\min\{2^{r/2},2^{b-r/2}\},
\qquad
h_{r,b}:=h(r,b,t_{r,b}),
\]
and apply the lemma with that value of $h_{r,b}$.

\begin{lemma}[Finite residue reduction]\label{p987:lem:tail-net-reduction}
There is an absolute constant $C_0>0$ with the following property.  Fix integers $r \ge 0$ and $b \ge 1$, and let $t>0$.  Define
\[
h=h(r,b,t):=\max\left\{0,\, b+1-\left\lfloor \log_2\!\left(\frac{t}{C_0}\right)\right\rfloor\right\},
\]
and for each integer $P\ge0$ let
\[
Q=Q_h(P):=P \bmod 2^h,
\qquad 0\le Q<2^h.
\]
Then for every integer $k\ge1$ with $b(k)=b$ one has
\[
|B_{P,r}(k)-B_{Q,r}(k)|\le t/2.
\]
Consequently, if
\[
|B_{Q,r}(k)|\le t/2
\]
for every integer $Q$ with $0\le Q<2^h$ and every integer $k\ge1$ with $b(k)=b$, then
\[
|B_{P,r}(k)|\le t
\]
for every integer $P\ge0$ and every integer $k\ge1$ with $b(k)=b$.
\end{lemma}

\begin{proof}
Fix an integer
\[
P=\sum_{\ell=0}^\infty p_\ell 2^\ell,
\]
and an integer $k\ge1$ with $b(k)=b$.  Write
\[
Q=\sum_{\ell=0}^\infty q_\ell 2^\ell,
\]
where $Q=Q_h(P)=P\bmod 2^h$.  Then $q_\ell=p_\ell$ for $0\le \ell<h$, while $q_\ell=0$ for $\ell\ge h$.
Then for every $w\in\bits^r$,
\[
\left|T_{w,P}-T_{w,Q}\right|\le 2^{-h}.
\]
Since $x\mapsto \e{x}$ is $2\pi$-Lipschitz,
\[
\left|
\e{\frac{k}{2^r}(j_r(w)+T_{w,P})}
-
\e{\frac{k}{2^r}(j_r(w)+T_{w,Q})}
\right|
\le
2\pi \frac{k}{2^r} 2^{-h}
\le
C_0 2^{b-r-h}.
\]
Summing over the $2^r$ values of $w$, this gives
\[
|B_{P,r}(k)-B_{Q,r}(k)|\le C_0 2^b 2^{-h}\le t/2.
\]
By hypothesis,
\[
|B_{Q,r}(k)|\le t/2.
\]
Therefore
\[
|B_{P,r}(k)|
\le
|B_{Q,r}(k)|+|B_{P,r}(k)-B_{Q,r}(k)|
\le
t.
\qedhere
\]
\end{proof}

\noindent In the proof of Proposition~\ref{p987:prop:block-bound}, Lemma~\ref{p987:lem:tail-net-reduction} is applied separately for each pair $(r,b)$, with
\[
t=t_{r,b}:=A\sqrt{r+b}\,\sigma,
\]
\[
h=h_{r,b}:=\max\left\{0,\, b+1-\left\lfloor \log_2\!\left(\frac{t_{r,b}}{C_0}\right)\right\rfloor\right\}.
\]
So there is no single global truncation length: the value of $h$ changes from scale to scale, and after the reduction its only effect is that only the $2^{h_{r,b}}$ residues modulo $2^{h_{r,b}}$ remain to be checked.

\begin{proof}[Proof of Proposition~\ref{p987:prop:block-bound}]
Fix $r \ge 0$ and $b \ge 1$.  Write
\[
M := r+b,
\qquad
\sigma := \min\{2^{r/2}, 2^{b-r/2}\},
\qquad
t := t_{r,b}:= A \sqrt{M}\,\sigma.
\]
\[
h_{r,b}:=\max\left\{0,\, b+1-\left\lfloor \log_2\!\left(\frac{t}{C_0}\right)\right\rfloor\right\}.
\]
Thus the present scale pair $(r,b)$ is assigned its own truncation length $h_{r,b}$.  If $A$ is large enough, then the probability that the stated estimate fails for this pair $(r,b)$ is summable over all $(r,b)$.

\smallskip
\noindent\textit{Step 1: finite residue reduction.}
Let $C_0$ be the constant from Lemma~\ref{p987:lem:tail-net-reduction}.  By the definition of $h_{r,b}$,
\[
C_0 2^b 2^{-h_{r,b}}\le t/2.
\]
Then
\begin{equation}\label{p987:eq:h-size}
2^{h_{r,b}} \le 1 + \frac{4C_0 2^b}{t}.
\end{equation}
	Let
\[
\mathcal R_{r,b}:=\{0,1,\dots,2^{h_{r,b}}-1\}.
\]
By Lemma~\ref{p987:lem:tail-net-reduction}, it is enough to prove
\[
|B_{Q,r}(k)|\le t/2
\]
for every $Q\in \mathcal R_{r,b}$ and every integer $k\ge1$ with $b(k)=b$.  In other words, Step 1 replaces the original supremum over all block locations $P\ge0$ by the finitely many residues modulo $2^{h_{r,b}}$.  After this reduction, the rest of the argument fixes one residue class $Q$ and proves a Bernstein bound for that fixed block sum.  The parameter $h_{r,b}$ will not reappear except through the cardinality bound
\begin{equation}\label{p987:eq:residue-count}
|\mathcal R_{r,b}|=2^{h_{r,b}} \le 1+\frac{4C_0 2^b}{t}.
\end{equation}

\smallskip
\noindent\textit{Step 2: conditioning on the short prefixes.}
Fix $Q \in \mathcal R_{r,b}$ and an integer $k\ge1$ with $b(k)=b$.  Write
\[
Q=\sum_{\ell=0}^\infty q_\ell 2^\ell,
\qquad q_\ell\in\bits,
\]
and let $\mathcal F_{<r}$ be the sigma-field generated by all $\eta_u$ with $|u| < r$.

For each $w \in \bits^r$, the random variable $T_{w,Q}$ depends on the bits
\[
\eta_w,\ \eta_{w q_0},\ \eta_{w q_0 q_1},\ \dots .
\]
These index sets are disjoint for different $w$, so the family $(T_{w,Q})_{w \in \bits^r}$ is conditionally independent given $\mathcal F_{<r}$.
Moreover each $T_{w,Q}$ is conditionally uniform on $[0,1)$, because its binary digits are independent fair bits.

Set
\[
R := 2^r,
\qquad
\theta := \frac{2 \pi k}{R},
\qquad
a := \E(e^{i \theta U}),
\]
where $U$ is uniform on $[0,1)$.  Define
\[
Y_w := e^{i \theta T_{w,Q}},
\qquad
Z_w := \e{\frac{k}{R} j_r(w)} (Y_w-a).
\]
Given $\mathcal F_{<r}$, the variables $Z_w$ are conditionally independent and satisfy $\mathbb E(Z_w\mid \mathcal F_{<r})=0$.

It is claimed that
\begin{equation}\label{p987:eq:sum-Z}
B_{Q,r}(k) = \sum_{w \in \bits^r} Z_w.
\end{equation}
Indeed,
\[
\sum_{w \in \bits^r} \e{\frac{k}{R}(j_r(w)+T_{w,Q})}
=
\sum_{w \in \bits^r} Z_w
+
a \sum_{w \in \bits^r} \e{\frac{k}{R} j_r(w)}.
\]
By Lemma~\ref{p987:lem:jr-bijection},
\[
\sum_{w \in \bits^r} \e{\frac{k}{R} j_r(w)}
=
\sum_{j=0}^{R-1} \e{\frac{k j}{R}}.
\]
This geometric sum vanishes unless $R \mid k$.  In the exceptional case $R \mid k$, one has $\theta \in 2\pi \mathbb{Z} \setminus \{0\}$, hence $a = \int_0^1 e^{i\theta u}\,du = 0$.  So \eqref{p987:eq:sum-Z} follows.

\smallskip
\noindent\textit{Step 3: bounds for the summands.}
If $r \le b$, then trivially $|Z_w| \le 2$.  If $r>b$, then $|\theta| < 2 \pi$, so
\[
|Y_w-a|
\le
|Y_w-1| + |1-a|
\le
|\theta| + \E|e^{i\theta U}-1|
\le
C |\theta|
\le
C 2^{b-r}.
\]
Thus in every case
\begin{equation}\label{p987:eq:Zw-bound}
|Z_w| \le C \min\{1,2^{b-r}\} = C \sigma 2^{-r/2}.
\end{equation}

Next, using an independent copy $U'$ of $U$,
\[
\operatorname{Var}(Y_w)
=
\frac12 \E |e^{i\theta U} - e^{i\theta U'}|^2.
\]
If $r \le b$, this is at most $1$.  If $r>b$, then $|\theta|<2\pi$ and
\[
|e^{i\theta U} - e^{i\theta U'}|
\le
|\theta|\, |U-U'|,
\]
so $\operatorname{Var}(Y_w) \le C \theta^2 \le C 2^{2(b-r)}$.  Therefore
\[
\sum_{w \in \bits^r} \E\bigl(|Z_w|^2 \mid \mathcal F_{<r}\bigr)
=
2^r \operatorname{Var}(Y_w)
\le
C \min\{2^r, 2^{2b-r}\}
=
C \sigma^2.
\]

\smallskip
\noindent\textit{Step 4: Bernstein's inequality.}
Apply Bernstein's inequality to the real and imaginary parts of $\sum_w Z_w$.  Using \eqref{p987:eq:Zw-bound} and the variance bound above gives
	\begin{equation}\label{p987:eq:bernstein}
	\Pbb\!\left(
	\left| B_{Q,r}(k) \right| \ge \frac{t}{2}
	\ \middle|\
	\mathcal F_{<r}
	\right)
\le
4 \exp\!\left(
- \frac{c t^2}{\sigma^2 + \sigma 2^{-r/2} t}
\right)
\le
4 \exp\!\left(
- \frac{c A^2 M}{1 + C A \sqrt{M}\, 2^{-r/2}}
\right).
\end{equation}
This bound is deterministic, so it also holds without conditioning on $\mathcal F_{<r}$. Next define the event
	\[
	E_{r,b}=\Big\{\text{there exist }Q\in\mathcal R_{r,b}\text{ and }k\ge1\text{ with }b(k)=b\text{ such that }|B_{Q,r}(k)|\ge t/2\Big\}.
	\]
	If $E_{r,b}$ does not occur, then the conclusion of Lemma~\ref{p987:lem:tail-net-reduction} holds, so the desired estimate
	\[
	|B_{P,r}(k)|\le t
	\]
	holds for every integer $P\ge0$ and every integer $k\ge1$ with $b(k)=b$.  To bound $\Pbb(E_{r,b})$, write
	\[
	\Pbb(E_{r,b})
	\le
	\sum_{\substack{k\ge1\\ b(k)=b}}\ \sum_{Q\in\mathcal R_{r,b}}
	\Pbb\bigl(|B_{Q,r}(k)|\ge t/2\bigr).
	\]
Now $h_{r,b}$ enters only through the number of residues.  Using \eqref{p987:eq:residue-count}, the Bernstein bound above, and the fact that there are fewer than $2^b$ such integers $k$, this gives
\begin{equation}\label{p987:eq:failure-probability}
\Pbb(E_{r,b})
\le
C 2^b \left(1 + \frac{2^b}{t}\right)
\exp\!\left(
- \frac{c A^2 M}{1 + C A \sqrt{M}\, 2^{-r/2}}
\right).
\end{equation}

\smallskip
\noindent\textit{Step 5: summing over $(r,b)$.}
We split the sum into two regimes.

\smallskip
\noindent\textbf{Regime I:} $A \sqrt{M}\, 2^{-r/2} \le 1$.
Then the denominator in the exponent in \eqref{p987:eq:failure-probability} is bounded by an absolute constant.  Moreover, if $r \le b$ then $2^b/t \le 2^b$, while if $r>b$ then
\[
\frac{2^b}{t} = \frac{2^{r/2}}{A\sqrt{M}} \le 2^{r/2}.
\]
Hence
\[
\log\!\left(2^b \left(1 + \frac{2^b}{t}\right)\right) \le C M.
\]
Hence
\[
\Pbb(E_{r,b}) \le e^{-(cA^2-C)M}.
\]
Since for each $M$ there are only $M$ pairs $(r,b)$ with $r+b=M$, the total contribution of Regime I is finite, and it can be made arbitrarily small by choosing $A$ large.

\smallskip
\noindent\textbf{Regime II:} $A \sqrt{M}\, 2^{-r/2} > 1$.
If $r \le b$, then $\sigma = 2^{r/2}$ and
\[
t = A \sqrt{M}\, 2^{r/2} > 2^r.
\]
Since $|B_{P,r}(k)| \le 2^r$ trivially, failure is impossible in this subcase.

It remains to consider the case $r>b$.  Here $M<2r$, and the inequality $A \sqrt{M}\, 2^{-r/2} > 1$ implies
\[
2^r < A^2 M < 2A^2 r.
\]
Hence $r = O(\log A)$, so there are only $O((\log A)^2)$ such pairs $(r,b)$.  For each of them, the prefactor in \eqref{p987:eq:failure-probability} is at most $e^{C r} = A^{O(1)}$, while the exponent satisfies
\[
\frac{c A^2 M}{1 + C A \sqrt{M}\, 2^{-r/2}}
\ge
c' A \sqrt{M}\, 2^{r/2}
\ge
c' A.
\]
Therefore the total contribution of this regime also tends to $0$ as $A \to \infty$.

Choosing $A$ sufficiently large makes
\[
\sum_{r \ge 0} \sum_{b \ge 1} \Pbb(E_{r,b}) < 1.
\]
So there is a realization of the random bits for which no event $E_{r,b}$ occurs.  Fix such a realization.  The argument above then yields, as desired:
	\[
	|B_{P,r}(k)| \ll \sqrt{r+b}\,\min\{2^{r/2},2^{b-r/2}\}.
  \qedhere
	\]
\end{proof}

\subsection{From dyadic blocks to arbitrary partial sums}

\begin{proof}[Proof of Theorem~\ref{p987:thm:main}]
Fix $k \ge 1$, and set
\[
b:=b(k)=\lceil \log_2(2k)\rceil.
\]
Let $L \ge 1$.  Write the binary expansion of $L$ as
\[
L = 2^{r_1} + 2^{r_2} + \cdots + 2^{r_s},
\qquad
r_1 > r_2 > \cdots \ge 0.
\]
Then the interval $[0,L)$ is a disjoint union of dyadic blocks
\[
I_j = [P_j 2^{r_j}, (P_j+1)2^{r_j})
\qquad
(1 \le j \le s)
\]
for suitable integers $P_j$.

By Proposition~\ref{p987:prop:block-identity},
\[
S_L(k)
=
\sum_{j=1}^s B_{P_j,r_j}(k).
\]
Applying Proposition~\ref{p987:prop:block-bound} and using that the exponents $r_j$ are distinct gives
\[
|S_L(k)|
\le
A \sum_{r=0}^\infty \sqrt{r+b}\, \min\{2^{r/2}, 2^{b-r/2}\}.
\]
Split the sum at $r=b$.  For $0 \le r \le b$,
\[
\sum_{r=0}^b \sqrt{r+b}\, 2^{r/2}
\le
\sqrt{2b}\sum_{r=0}^b 2^{r/2}
\ll
2^{b/2}\sqrt{b}.
\]
For $r=b+s$ with $s \ge 1$,
\[
\sum_{s=1}^\infty \sqrt{2b+s}\, 2^{b/2} 2^{-s/2}
\ll
2^{b/2}\sqrt{b}\sum_{s=1}^\infty (1+s)^{1/2} 2^{-s/2}
\ll
2^{b/2}\sqrt{b}.
\]
Hence
\[
|S_L(k)| \ll 2^{b/2}\sqrt{b}.
\]
Since $2^{b/2} \le \sqrt{2k}$ and $b \le \log_2(2k)$, this gives
\[
|S_L(k)| \ll \sqrt{k \log(2k)}.
\]
Since the bound is uniform in $L$, this is exactly the statement of the theorem.
\end{proof}

\section[Problem 1091]{Chord-bounded 4-chromatic graphs with all small subgraphs 3-colorable}\label{sec:1091}

This problem asks for graphs with chromatic number $4$, such that all small subgraphs are $3$-colorable, and all odd-length subcycles have a bounded number of chords.
In fact any proper subgraph of our construction is $2$-degenerate, and every cycle of even or odd length has at most $10$ chords.
The construction uses a caterpillar graph composed of pentagonal blocks.


\subsection{Introduction}

For a graph $G$ and a cycle $C\subseteq G$, write $\ch_G(C)$ for the number of edges of $G$ joining two nonconsecutive vertices of $C$.
These edges are often called \emph{chords} or \emph{diagonals} of $C$.

The classical starting point is the theorem of Voss \cite{Vo82}, building on Larson \cite{La79}, that every $K_4$-free $4$-chromatic graph contains an odd cycle with at least two chords.
A natural quantitative strengthening, formulated appearing as \cite[Problem~1091]{BloWeb}, asks:

\medskip
\noindent
\emph{Does there exist a function $f(r)\to\infty$ such that every $4$-chromatic graph $G$ for which every subgraph on at most $r$ vertices is $3$-colorable contains an odd cycle $C$ with $\ch_G(C)\ge f(r)$?}
\medskip

The purpose of this section is to show that the answer is no.
In fact the construction below gives an explicit family of $K_4$-free counterexamples.

\begin{theorem}\label{p1091:thm:main}
For every integer $m\ge1$ there exists an explicit $K_4$-free graph $G_m$ on
$20m+31$ vertices such that:
\begin{enumerate}
    \item $\chi(G_m)=4$;
    \item every proper subgraph $H\subsetneq G_m$ is $3$-colorable (in fact $2$-degenerate);
    \item every cycle $C$ in $G_m$ satisfies
    $
    \ch_{G_m}(C)\le10.
    $
\end{enumerate}
\end{theorem}

The graph is built as a caterpillar of pentagonal ``blocks'' and an additional vertex $v$ which is connected to any vertex of degree $2$ in the original construction. The fact that $G$ has chromatic number $4$ can be proven by considering the color of $v$ and propagating how it forces each pentagon to be colored.

Throughout the section, an \emph{inter-block edge} means an edge joining two different pentagon blocks in the graph $G_m^0$; the edges from $v$ to the leaf blocks are not inter-block edges.
Figure~\ref{p1091:fig:local-gadget} shows the local structure: each block is a $5$-cycle with labelled vertices, and each leaf block is connected to a unique spine block by one such inter-block edge.
A single extra vertex $v$ is adjacent to the four non-spine-adjacent vertices of every leaf block.
Figure~\ref{p1091:fig:block-tree} shows the global structure: the spine blocks form a path $S_0,S_1,\dots,S_m$, and the leaf blocks are attached to specific labelled vertices of these spine blocks.

\subsection{The construction}

Fix $m\ge1$.
We first describe the \emph{block tree}.
Its spine blocks are
\[
S_0,S_1,\dots,S_m.
\]
Attach to $S_0$ the four leaf blocks
\[
L_{0,a},L_{0,b},L_{0,d},L_{0,e},
\]
to each internal spine block $S_i$ with $1\le i\le m-1$ attach three leaf blocks
\[
L_{i,b},L_{i,d},L_{i,e},
\]
and to $S_m$ attach four leaf blocks
\[
L_{m,b},L_{m,c},L_{m,d},L_{m,e}.
\]
Thus the block tree consists of
$
(m+1)+4+3(m-1)+4=4m+6
$
total blocks (including both spine and leaf).
Each spine block $S_i$ is then replaced by a $5$-cycle whose vertices are labelled cyclically by
\[
a,b,c,d,e,
\]
and each leaf block $L_{i,x}$ is replaced by a $5$-cycle whose vertices are labelled cyclically by
\[
A,B,C,D,E.
\]
These labels are local to each block.
For the spine blocks, write $S_i[x]$ with $x\in\{a,b,c,d,e\}$, and for the leaf blocks, write $L_{i,x}[Y]$ with $x\in\{a,b,c,d,e\}$ and $Y\in\{A,B,C,D,E\}$.

Whenever $x\in\{a,b,c,d,e\}$, let $X$ denote the same letter in uppercase.
Then the leaf block $L_{i,x}$ is attached to $S_i$ by the inter-block edge
\[
S_i[x]L_{i,x}[X].
\]
For instance, $L_{i,e}$ is attached by the inter-block edge $S_i[e]L_{i,e}[E]$.
Along the spine, join $S_i[c]$ to $S_{i+1}[a]$ for each $0\le i<m$; these spine-to-spine edges are also inter-block edges.
Finally add one extra vertex $v$ and, for every leaf block $L_{i,x}$, join $v$ to the four vertices $L_{i,x}[Y]$ with $Y\neq X$.
These edges incident to $v$ are ordinary edges of $G_m$, but they are not inter-block edges.

With these attachment rules, every vertex in a leaf pentagon has exactly one neighbor outside that pentagon: the attachment vertex $L_{i,x}[X]$ is joined to $S_i[x]$, while each of the other four vertices $L_{i,x}[Y]$ is joined to $v$.
Likewise every vertex in a spine pentagon is incident to exactly one inter-block edge, either to a neighboring spine block or to a leaf block.
Hence every vertex of $G_m$ except $v$ has degree exactly $3$.

Denote the resulting graph by $G_m$, and write
\[
G_m^0:=G_m\backslash\{v\}.
\]
Thus $G_m^0$ is a tree of pentagons joined by inter-block edges.

\begin{figure}[!htbp]
\centering
\begin{tikzpicture}[
  scale=
  1
  ,transform shape,
    blockv/.style={circle,draw,fill=white,minimum size=4.4mm,inner sep=0pt,font=\scriptsize},
    vtx/.style={ellipse,draw,fill=red!8,minimum width=13mm,minimum height=6mm,inner sep=1pt,font=\small},
    spineedge/.style={line width=1pt,blue!70!black},
    vedge/.style={line width=0.6pt,red!70!black,densely dashed,opacity=0.45},
    callout/.style={font=\scriptsize,text=black,fill=white,inner sep=1pt}
]
\begin{scope}[shift={(-4.7,3.55)}]
  \node[blockv] (Pa) at (90:1.0) {$a$};
  \node[blockv] (Pb) at (18:1.0) {$b$};
  \node[blockv] (Pc) at (-54:1.0) {$c$};
  \node[blockv] (Pd) at (-126:1.0) {$d$};
  \node[blockv] (Pe) at (162:1.0) {$e$};
  \draw[line width=1pt,fill=blue!8] (Pa.center)--(Pb.center)--(Pc.center)--(Pd.center)--(Pe.center)--cycle;
  \foreach \n/\lab in {Pa/a,Pb/b,Pc/c,Pd/d,Pe/e} \node[blockv] at (\n) {$\lab$};
  \node[font=\small] at (0,-1.65) {$S_{i-1}$};
\end{scope}

\begin{scope}[shift={(0,2.65)}]
  \node[blockv] (Sa) at (90:1.2) {$a$};
  \node[blockv] (Sb) at (18:1.2) {$b$};
  \node[blockv] (Sc) at (-54:1.2) {$c$};
  \node[blockv] (Sd) at (-126:1.2) {$d$};
  \node[blockv] (Se) at (162:1.2) {$e$};
  \draw[line width=1pt,fill=blue!8] (Sa.center)--(Sb.center)--(Sc.center)--(Sd.center)--(Se.center)--cycle;
  \node[blockv] at (Sa) {$a$};
  \node[blockv] at (Sb) {$b$};
  \node[blockv] at (Sc) {$c$};
  \node[blockv] at (Sd) {$d$};
  \node[blockv] at (Se) {$e$};
  \node[font=\small] at (0,-1.95) {$S_i$};
\end{scope}

\begin{scope}[shift={(4.7,3.55)}]
  \node[blockv] (Na) at (90:1.0) {$a$};
  \node[blockv] (Nb) at (18:1.0) {$b$};
  \node[blockv] (Nc) at (-54:1.0) {$c$};
  \node[blockv] (Nd) at (-126:1.0) {$d$};
  \node[blockv] (Ne) at (162:1.0) {$e$};
  \draw[line width=1pt,fill=blue!8] (Na.center)--(Nb.center)--(Nc.center)--(Nd.center)--(Ne.center)--cycle;
  \foreach \n/\lab in {Na/a,Nb/b,Nc/c,Nd/d,Ne/e} \node[blockv] at (\n) {$\lab$};
  \node[font=\small] at (0,-1.65) {$S_{i+1}$};
\end{scope}

\begin{scope}[shift={(-4.7,-0.45)}]
  \node[blockv] (Ea) at (90:1.2) {$A$};
  \node[blockv] (Eb) at (18:1.2) {$B$};
  \node[blockv] (Ec) at (-54:1.2) {$C$};
  \node[blockv] (Ed) at (-126:1.2) {$D$};
  \node[blockv] (Ee) at (162:1.2) {$E$};
  \draw[line width=1pt,fill=orange!12] (Ea.center)--(Eb.center)--(Ec.center)--(Ed.center)--(Ee.center)--cycle;
  \foreach \n/\lab in {Ea/A,Eb/B,Ec/C,Ed/D,Ee/E} \node[blockv] at (\n) {$\lab$};
  \node[font=\small,anchor=east] at (-1.75,-0.1) {$L_{i,e}$};
\end{scope}

\begin{scope}[shift={(0,-1.1)}]
  \node[blockv] (Da) at (90:1.2) {$A$};
  \node[blockv] (Db) at (18:1.2) {$B$};
  \node[blockv] (Dc) at (-54:1.2) {$C$};
  \node[blockv] (Dd) at (-126:1.2) {$D$};
  \node[blockv] (De) at (162:1.2) {$E$};
  \draw[line width=1pt,fill=orange!12] (Da.center)--(Db.center)--(Dc.center)--(Dd.center)--(De.center)--cycle;
  \foreach \n/\lab in {Da/A,Db/B,Dc/C,Dd/D,De/E} \node[blockv] at (\n) {$\lab$};
  \node[font=\small] at (0,-1.85) {$L_{i,d}$};
\end{scope}

\begin{scope}[shift={(4.7,-0.45)}]
  \node[blockv] (Ba) at (90:1.2) {$A$};
  \node[blockv] (Bb) at (18:1.2) {$B$};
  \node[blockv] (Bc) at (-54:1.2) {$C$};
  \node[blockv] (Bd) at (-126:1.2) {$D$};
  \node[blockv] (Be) at (162:1.2) {$E$};
  \draw[line width=1pt,fill=orange!12] (Ba.center)--(Bb.center)--(Bc.center)--(Bd.center)--(Be.center)--cycle;
  \foreach \n/\lab in {Ba/A,Bb/B,Bc/C,Bd/D,Be/E} \node[blockv] at (\n) {$\lab$};
  \node[font=\small,anchor=west] at (1.75,-0.1) {$L_{i,b}$};
\end{scope}

\draw[spineedge] (Pc)
  .. controls ($(Pc)+(2.0,0.5)$) and ($(Sa)+(-0.4,1.8)$) ..
  node[callout,above,pos=.48,yshift=2pt] {$S_{i-1}[c]S_i[a]$} (Sa);
\draw[spineedge] (Sc)
  .. controls ($(Sc)+(1.5,0.45)$) and ($(Na)+(-2.0,0.5)$) ..
  node[callout,above,pos=.52,yshift=2pt] {$S_i[c]S_{i+1}[a]$} (Na);
\draw[spineedge] (Se) to[out=220,in=15]
  node[callout,above,pos=.22,xshift=-4pt,yshift=3pt] {$S_i[e]L_{i,e}[E]$} (Ee);
\draw[spineedge] (Sd) -- node[callout,left,pos=.42,xshift=-2pt] {$S_i[d]L_{i,d}[D]$} (Dd);
\draw[spineedge] (Sb) to[out=-40,in=165]
  node[callout,above,pos=.22,xshift=4pt,yshift=3pt] {$S_i[b]L_{i,b}[B]$} (Bb);

\node[vtx] (v) at (0,-4.25) {$v$};
\coordinate (vLone) at ($(v.north west)!0.08!(v.north east)$);
\coordinate (vLtwo) at ($(v.north west)!0.18!(v.north east)$);
\coordinate (vLthree) at ($(v.north west)!0.28!(v.north east)$);
\coordinate (vLfour) at ($(v.north west)!0.38!(v.north east)$);
\coordinate (vCone) at ($(v.north west)!0.45!(v.north east)$);
\coordinate (vCtwo) at ($(v.north west)!0.50!(v.north east)$);
\coordinate (vCthree) at ($(v.north west)!0.55!(v.north east)$);
\coordinate (vCfour) at ($(v.north west)!0.60!(v.north east)$);
\coordinate (vRone) at ($(v.north west)!0.67!(v.north east)$);
\coordinate (vRtwo) at ($(v.north west)!0.77!(v.north east)$);
\coordinate (vRthree) at ($(v.north west)!0.87!(v.north east)$);
\coordinate (vRfour) at ($(v.north west)!0.97!(v.north east)$);
\draw[vedge] (vLone) -- (Ed.south);
\draw[vedge] (vLtwo) -- (Ea.south);
\draw[vedge] (vLthree) -- (Ec.south);
\draw[vedge] (vLfour) -- (Eb.south);
\draw[vedge] (vCone) to[out=108,in=-120] (De.south);
\draw[vedge] (vCtwo) -- (Da.south);
\draw[vedge] (vCthree) -- (Dc.south);
\draw[vedge] (vCfour) to[out=72,in=-60] (Db.south);
\draw[vedge] (vRone) -- (Be.south);
\draw[vedge] (vRtwo) -- (Bd.south);
\draw[vedge] (vRthree) -- (Ba.south);
\draw[vedge] (vRfour) -- (Bc.south);
\end{tikzpicture}
\caption{
  The local attachment rules from Section~\ref{sec:1091}, shown around one internal spine block $S_i$.
The two adjacent spine blocks $S_{i-1}$ and $S_{i+1}$ and all three leaf blocks $L_{i,e},L_{i,d},L_{i,b}$ are included.
The spine pentagon uses local labels $a,b,c,d,e$, while each leaf pentagon uses local labels $A,B,C,D,E$.
Blue edges are inter-block edges; red edges connect the single special vertex $v$ to every leaf-block vertex except for the spine-adjacent ``attachment'' vertices, e.g. $L_{i,e}[E]$, $L_{i,d}[D]$, $L_{i,b}[B]$.
Thus every vertex of $G_m$ except $v$ has degree exactly $3$.
In Lemma~\ref{p1091:lem:leaf-forcing} we observe that in any putative $3$-coloring of $G_m$, the red edges force each spine-adjacent leaf vertex to have the same color as $v$.
Lemma~\ref{p1091:lem:propagation} then propagates vertices matching the color of $v$ down the $a,c$ spine, contradicting $3$-colorability in Proposition~\ref{p1091:prop:not-3-colorable}.
}
\label{p1091:fig:local-gadget}
\end{figure}
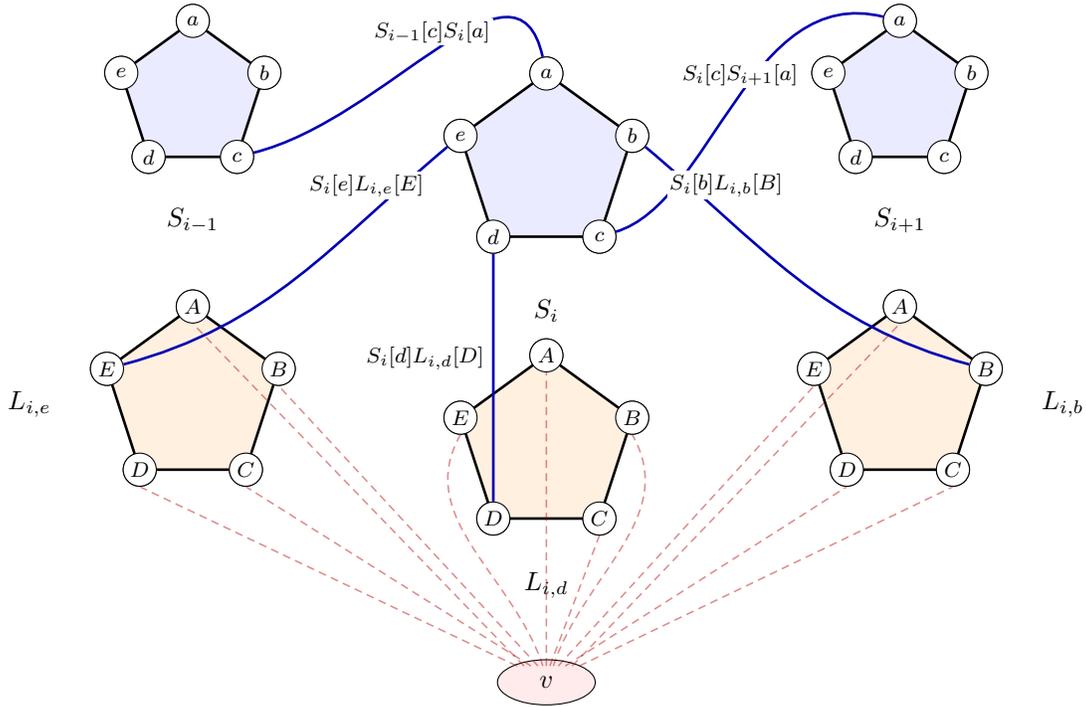

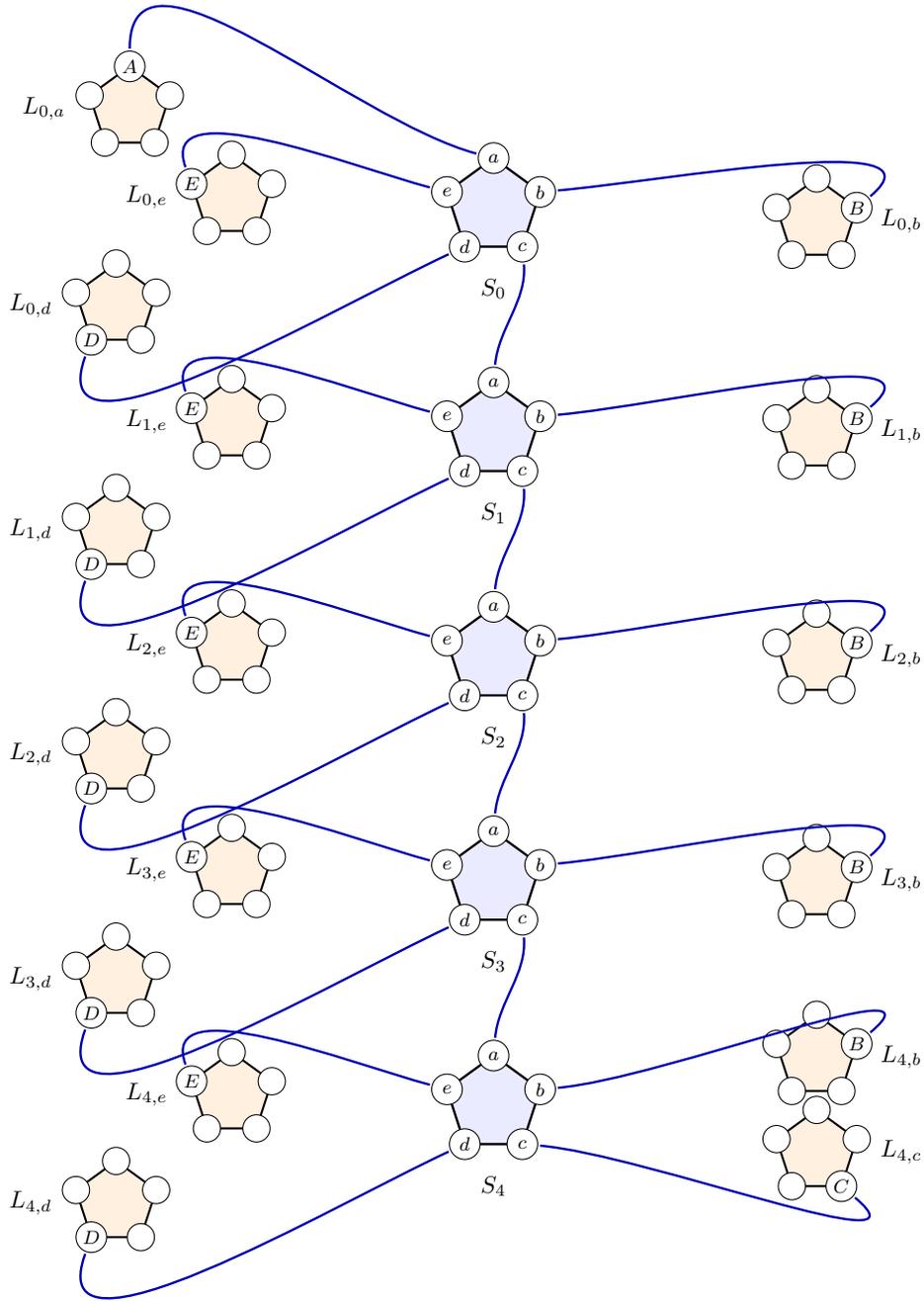
\begin{figure}[t]
\centering
\begin{tikzpicture}[scale=0.915,transform shape,
    spinev/.style={circle,draw,fill=white,minimum size=4.5mm,inner sep=0pt,font=\scriptsize},
    leafv/.style={circle,draw,fill=white,minimum size=4.0mm,inner sep=0pt},
    leafmark/.style={circle,draw,fill=white,minimum size=4.5mm,inner sep=0pt,font=\scriptsize},
    edge/.style={line width=0.9pt,blue!70!black,shorten <=6.8pt,shorten >=6.8pt},
    blockname/.style={font=\small,fill=white,inner sep=1pt},
    spinepent/.pic={
      \coordinate (-a) at (90:0.72);
      \coordinate (-b) at (18:0.72);
      \coordinate (-c) at (-54:0.72);
      \coordinate (-d) at (-126:0.72);
      \coordinate (-e) at (162:0.72);
      \draw[line width=0.9pt,fill=blue!8]
        (-a)--(-b)--(-c)--(-d)--(-e)--cycle;
      \node[spinev] at (-a) {$a$};
      \node[spinev] at (-b) {$b$};
      \node[spinev] at (-c) {$c$};
      \node[spinev] at (-d) {$d$};
      \node[spinev] at (-e) {$e$};
    },
    leafpent/.pic={
      \coordinate (-A) at (90:0.62);
      \coordinate (-B) at (18:0.62);
      \coordinate (-C) at (-54:0.62);
      \coordinate (-D) at (-126:0.62);
      \coordinate (-E) at (162:0.62);
      \draw[line width=0.8pt,fill=orange!12]
        (-A)--(-B)--(-C)--(-D)--(-E)--cycle;
      \node[leafv] at (-A) {};
      \node[leafv] at (-B) {};
      \node[leafv] at (-C) {};
      \node[leafv] at (-D) {};
      \node[leafv] at (-E) {};
    }
]
\pic (Szero) at (0,13.2) {spinepent};
\pic (Sone) at (0,9.9) {spinepent};
\pic (Stwo) at (0,6.6) {spinepent};
\pic (Sthree) at (0,3.3) {spinepent};
\pic (Sfour) at (0,0) {spinepent};

\node[blockname] at (0,12.0) {$S_0$};
\node[blockname] at (0,8.7) {$S_1$};
\node[blockname] at (0,5.4) {$S_2$};
\node[blockname] at (0,2.1) {$S_3$};
\node[blockname] at (0,-1.2) {$S_4$};

\pic (Lza) at (-5.35,14.65) {leafpent};
\pic (Lze) at (-3.85,13.35) {leafpent};
\pic (Lzd) at (-5.55,11.75) {leafpent};
\pic (Lzb) at (4.75,13.0) {leafpent};
\node[blockname,anchor=east] at (-6.25,14.65) {$L_{0,a}$};
\node[blockname,anchor=east] at (-4.75,13.35) {$L_{0,e}$};
\node[blockname,anchor=east] at (-6.45,11.75) {$L_{0,d}$};
\node[blockname,anchor=west] at (5.65,13.0) {$L_{0,b}$};
\node[leafmark] at (Lza-A) {$A$};
\node[leafmark] at (Lze-E) {$E$};
\node[leafmark] at (Lzd-D) {$D$};
\node[leafmark] at (Lzb-B) {$B$};
\draw[edge] (Lza-A) to[out=90,in=165] (Szero-a);
\draw[edge] (Lze-E) to[out=110,in=165] (Szero-e);
\draw[edge] (Lzd-D) to[out=250,in=205] (Szero-d);
\draw[edge] (Lzb-B) to[out=40,in=5] (Szero-b);

\pic (Loe) at (-3.85,10.05) {leafpent};
\pic (Lod) at (-5.55,8.45) {leafpent};
\pic (Lob) at (4.75,9.9) {leafpent};
\node[blockname,anchor=east] at (-4.75,10.05) {$L_{1,e}$};
\node[blockname,anchor=east] at (-6.45,8.45) {$L_{1,d}$};
\node[blockname,anchor=west] at (5.65,9.9) {$L_{1,b}$};
\node[leafmark] at (Loe-E) {$E$};
\node[leafmark] at (Lod-D) {$D$};
\node[leafmark] at (Lob-B) {$B$};
\draw[edge] (Loe-E) to[out=110,in=165] (Sone-e);
\draw[edge] (Lod-D) to[out=250,in=205] (Sone-d);
\draw[edge] (Lob-B) to[out=40,in=5] (Sone-b);

\pic (Lte) at (-3.85,6.75) {leafpent};
\pic (Ltd) at (-5.55,5.15) {leafpent};
\pic (Ltb) at (4.75,6.6) {leafpent};
\node[blockname,anchor=east] at (-4.75,6.75) {$L_{2,e}$};
\node[blockname,anchor=east] at (-6.45,5.15) {$L_{2,d}$};
\node[blockname,anchor=west] at (5.65,6.6) {$L_{2,b}$};
\node[leafmark] at (Lte-E) {$E$};
\node[leafmark] at (Ltd-D) {$D$};
\node[leafmark] at (Ltb-B) {$B$};
\draw[edge] (Lte-E) to[out=110,in=165] (Stwo-e);
\draw[edge] (Ltd-D) to[out=250,in=205] (Stwo-d);
\draw[edge] (Ltb-B) to[out=40,in=5] (Stwo-b);

\pic (Lhe) at (-3.85,3.45) {leafpent};
\pic (Lhd) at (-5.55,1.85) {leafpent};
\pic (Lhb) at (4.75,3.3) {leafpent};
\node[blockname,anchor=east] at (-4.75,3.45) {$L_{3,e}$};
\node[blockname,anchor=east] at (-6.45,1.85) {$L_{3,d}$};
\node[blockname,anchor=west] at (5.65,3.3) {$L_{3,b}$};
\node[leafmark] at (Lhe-E) {$E$};
\node[leafmark] at (Lhd-D) {$D$};
\node[leafmark] at (Lhb-B) {$B$};
\draw[edge] (Lhe-E) to[out=110,in=165] (Sthree-e);
\draw[edge] (Lhd-D) to[out=250,in=205] (Sthree-d);
\draw[edge] (Lhb-B) to[out=40,in=5] (Sthree-b);

\pic (Lfe) at (-3.85,0.15) {leafpent};
\pic (Lfd) at (-5.55,-1.45) {leafpent};
\pic (Lfb) at (4.75,0.7) {leafpent};
\pic (Lfc) at (4.75,-0.7) {leafpent};
\node[blockname,anchor=east] at (-4.75,0.15) {$L_{4,e}$};
\node[blockname,anchor=east] at (-6.45,-1.45) {$L_{4,d}$};
\node[blockname,anchor=west] at (5.65,0.7) {$L_{4,b}$};
\node[blockname,anchor=west] at (5.65,-0.7) {$L_{4,c}$};
\node[leafmark] at (Lfe-E) {$E$};
\node[leafmark] at (Lfd-D) {$D$};
\node[leafmark] at (Lfb-B) {$B$};
\node[leafmark] at (Lfc-C) {$C$};
\draw[edge] (Lfe-E) to[out=110,in=165] (Sfour-e);
\draw[edge] (Lfd-D) to[out=250,in=205] (Sfour-d);
\draw[edge] (Lfb-B) to[out=40,in=5] (Sfour-b);
\draw[edge] (Lfc-C) to[out=-40,in=-5] (Sfour-c);

\draw[edge] (Szero-c) to[out=-85,in=85] (Sone-a);
\draw[edge] (Sone-c) to[out=-85,in=85] (Stwo-a);
\draw[edge] (Stwo-c) to[out=-85,in=85] (Sthree-a);
\draw[edge] (Sthree-c) to[out=-85,in=85] (Sfour-a);
\end{tikzpicture}
\caption{The graph $G_m^0=G_m\backslash \{v\}$ is shown with five spine blocks, i.e. $m=4$.
Each spine pentagon has vertices labelled $a,b,c,d,e$.
For each leaf pentagon, the only label displayed is for the attachment vertex $X\in\{A,B,C,D,E\}$ used by its inter-block edge $S_i[x]L_{i,x}[X]$ connecting it to the spine.
The only endpoint asymmetry is that $S_0$ has no $c$-leaf, while $S_m$ has no $a$-leaf; this is key in Lemma~\ref{p1091:lem:propagation} and Proposition~\ref{p1091:prop:not-3-colorable}.
In the graph $G_m$, the additional special vertex $v$ is connected to all \emph{unlabelled} leaf vertices (i.e. those not connected to a spine block).
Thus all vertices shown above have degree exactly $3$ in $G_m$.
}
\label{p1091:fig:block-tree}
\end{figure}

\clearpage


\begin{lemma}\label{p1091:lem:k4free}
The graph $G_m$ is $K_4$-free.
\end{lemma}

\begin{proof}
The graph $G_m^0=G_m\backslash\{v\}$ is triangle-free: each block is a $5$-cycle, and different blocks are connected only by inter-block edges.
So any $K_4$ in $G_m$ would have to contain $v$.
In a leaf block $L_{i,x}$, if $X$ is the uppercase version of $x$, then the neighbors of $v$ are the four vertices $L_{i,x}[Y]$ with $Y\neq X$, and they induce the $4$-vertex path obtained from that pentagon by deleting $L_{i,x}[X]$.
Hence $N(v)$ is a disjoint union of paths, so in particular it is triangle-free.
Thus $v$ cannot lie in a $K_4$ either.
\end{proof}

\subsection{\texorpdfstring{$G_m$}{Gm} is not \texorpdfstring{$3$-colorable}{3-colorable}}

Assume for contradiction that $G_m$ has a proper $3$-coloring.
Let $\alpha$ be the color of $v$, and let $\beta,\gamma$ be the other two colors.

\begin{lemma}[Leaf forcing]\label{p1091:lem:leaf-forcing}
In every leaf block $L_{i,x}$, if $X$ is the uppercase version of $x$, then the attachment vertex $L_{i,x}[X]$ has the same color as $v$.
\end{lemma}

\begin{proof}
Remove the attachment vertex $L_{i,x}[X]$ from the leaf pentagon.
The other four leaf vertices form a path, and all four are adjacent to $v$, so they can use only the colors $\beta$ and $\gamma$.
Hence they must alternate along that path.
In particular, the two endpoints of the path, which are exactly the two neighbors of $L_{i,x}[X]$ inside the pentagon, receive different colors.
Therefore $L_{i,x}[X]$ cannot use $\beta$ or $\gamma$, so it must use $\alpha$.
\end{proof}

\begin{lemma}
\label{p1091:lem:propagation}
Let $Q$ be a spine pentagon.
\begin{enumerate}
    \item If $Q=S_0$ and $Q[a],Q[b],Q[d],Q[e]$ all avoid the color $\alpha$, then $Q[c]$ has color $\alpha$.
    \item If $Q=S_i$ with $1\le i\le m-1$ and $Q[a],Q[b],Q[d],Q[e]$ all avoid the color $\alpha$, then $Q[c]$ has color $\alpha$.
    \item If $Q=S_m$ and $Q[b],Q[c],Q[d],Q[e]$ all avoid the color $\alpha$, then $Q[a]$ has color $\alpha$.
\end{enumerate}
\end{lemma}

\begin{proof}
For (1) and (2), the same argument applies.
Since $Q[a]$ avoids $\alpha$, it has color $\beta$ or $\gamma$.
Its two neighbors $Q[b]$ and $Q[e]$ therefore both have the other non-$\alpha$ color.
Then $Q[d]$, being adjacent to $Q[e]$ and also avoiding $\alpha$, has the same color as $Q[a]$.
So $Q[b]$ and $Q[d]$ have different colors, and their common neighbor $Q[c]$ must use the third color $\alpha$.

For (3), the path $Q[b]-Q[c]-Q[d]-Q[e]$ uses only the colors $\beta$ and $\gamma$, so it alternates.
Hence $Q[b]$ and $Q[e]$ have different colors.
Since $Q[a]$ is adjacent to both $Q[b]$ and $Q[e]$, the only available color for $Q[a]$ is $\alpha$.
\end{proof}

\begin{proposition}\label{p1091:prop:not-3-colorable}
The graph $G_m$ is not $3$-colorable.
\end{proposition}

\begin{proof}
Assume that a proper $3$-coloring exists.
By Lemma~\ref{p1091:lem:leaf-forcing}, every leaf attachment vertex has color $\alpha$.
Consequently, every spine vertex incident to a leaf inter-block edge avoids the color $\alpha$.

Apply Lemma~\ref{p1091:lem:propagation}(1) to $S_0$.
The vertices $S_0[a],S_0[b],S_0[d],S_0[e]$ all avoid $\alpha$, so $S_0[c]$ has color $\alpha$.
Since $S_0[c]$ is adjacent to $S_1[a]$, the latter avoids $\alpha$.

Now suppose $1\le i\le m-1$ and $S_i[a]$ avoids $\alpha$.
The vertices $S_i[b],S_i[d],S_i[e]$ also avoid $\alpha$, because they are incident to leaf inter-block edges.
So Lemma~\ref{p1091:lem:propagation}(2) gives $S_i[c]=\alpha$.
Hence the inter-block edge to the next spine block forces $S_{i+1}[a]\neq\alpha$.
By induction,
\[
S_i[c]=\alpha \qquad (0\le i\le m-1),
\]
and therefore $S_m[a]\neq\alpha$.

In the terminal block $S_m$, all four slots
$b,c,d,e$ avoid $\alpha$.
Lemma~\ref{p1091:lem:propagation}(3) therefore implies
$S_m[a]=\alpha$.
This contradiction shows that no proper $3$-coloring exists.
\end{proof}

\subsection{All Proper Subgraphs of $G_m$ are \texorpdfstring{$2$-degenerate}{2-degenerate}}

Recall that a graph is \emph{$2$-degenerate} if every nonempty subgraph has a vertex of degree at most $2$.
Equivalently, this means its vertices can be removed one by one (``peeled'') so that each removed vertex has degree at most $2$ in the current graph at the time of its removal.
Note that any $2$-degenerate graph is $3$-colorable by induction, by assigning colors greedily in the opposite order of peeling.

\begin{proposition}\label{p1091:prop:degenerate}
For every edge $e\in E(G_m)$, the graph $G_m\backslash\{e\}$ is $2$-degenerate, hence $3$-colorable.
\end{proposition}

\begin{proof}
We show that any non-empty proper subgraph $H\subsetneq G_m$ has minimum degree at most $2$.
This follows from the fact that $G_m^0=G_m\backslash\{v\}$ is connected and all vertices except $v$ have degree $3$.
Indeed, first suppose that $V(H)=V(G_m)$; then consider $e\in E(H)\backslash E(G_m)$ and note that $e$ is incident to a vertex different from $v$; said vertex has $H$-degree at most $2$.
Next suppose $V(H)\subsetneq V(G_m)$.
If $V(H)\subseteq \{v\}$ the conclusion is trivial.
If not, since $G_m^0$ is connected, there is $u\in V(H)\backslash \{v\}$ with at least one $G_m$-neighbor not in $V(H)$.
Then $u$ has $H$-degree at most $2$.
This completes the proof.
\end{proof}

\subsection{Bounding the number of chords}

It remains to show that every cycle in $G_m$ has a uniformly bounded number of chords.
First, only cycles including $v$ can have chords; then, given a cycle containing $v$, the chords are counted separately for spine blocks and leaf blocks (no chord can connect a leaf vertex to a spine vertex).

\begin{lemma}\label{p1091:lem:cycle-in-base}
Every cycle in $G_m^0$ is contained in a single pentagon block.
In particular, every cycle in $G_m^0$ has no chords.
\end{lemma}

\begin{proof}
Every inter-block edge of $G_m^0$ is a cut edge, so no cycle can use one.
Hence any cycle of $G_m^0$ lies inside one block, and each block is a $5$-cycle.
\end{proof}

\begin{lemma}\label{p1091:lem:endpoint-leaf}
Let $C$ be a cycle containing the special vertex $v$, let $Q=L_{i,x}$ be a leaf block visited by $C$, and $X$ the uppercase version of $x$.
Count a chord of $C$ \emph{in $Q$} if either both endpoints lie in $V(Q)$, or one endpoint is $v$ and the other lies in $V(Q)$.
Then at most $4$ chords of $C$ are counted in $Q$.
\end{lemma}

\begin{proof}
The intersection $P:=C\cap Q$ is a rim path in the pentagon $Q$.
Its two endpoints are the attachment vertex $Q[X]$ and a vertex $Q[Y]$ with $Y\neq X$ that is joined to $v$ by an edge of $C$.
The only possible chords of $C$ counted in $Q$ are:
\begin{enumerate}
    \item edges from $v$ to internal vertices of $P$, of which there are at most $3$;
    \item possibly the pentagon edge $Q[X]Q[Y]$, which can only be a chord when $Q[X]$ and $Q[Y]$ are adjacent on the rim and $P$ is the longer of the two rim paths between them.
\end{enumerate}
So at most $4$ chords of $C$ are counted in $Q$.
\end{proof}

\begin{lemma}\label{p1091:lem:spine-contrib}
Let $C$ be a cycle containing the special vertex $v$.
For a spine block $Q$ visited by $C$, count a chord of $C$ \emph{in $Q$} if both endpoints lie in $V(Q)$.
Then:
\begin{enumerate}
    \item no chord of $C$ is counted in an internal spine block;
    \item $\leq 1$ chord of $C$ is counted in each of the two end spine blocks on the visited spine segment.
\end{enumerate}
\end{lemma}

\begin{proof}
Let $Q$ be a visited spine block, and let $P:=C\cap Q$.
Then $P$ is a rim path in the pentagon $Q$, and every chord counted in $Q$ must be a pentagon edge of $Q$ joining the two endpoints of $P$ while not itself lying on $P$.

If $Q=S_i$ is internal on the visited spine segment, then the endpoints of $P$ are $S_i[a]$ and $S_i[c]$.
These are not adjacent on the pentagon, so no chord of $C$ is counted in $Q$.

If $Q$ is one of the two end spine blocks on the visited spine segment, let $p,q$ be the two endpoints of $P$.
Regardless of which inter-block edges are used to enter and leave $Q$, the only possible chord counted in $Q$ is the pentagon edge $pq$.
This can occur only when $p$ and $q$ are adjacent on the rim and $P$ is the longer rim path between them, yielding at most $1$ chord.
\end{proof}

\begin{proposition}\label{p1091:prop:chord-bound}
Every cycle $C$ in $G_m$ satisfies $\ch_{G_m}(C)\le 10$.
\end{proposition}

\begin{proof}
If $v\notin V(C)$, then Lemma~\ref{p1091:lem:cycle-in-base} shows that $C$ lies in one pentagon and has no chords.

Now assume that $v\in V(C)$.
Then $C\backslash\{v\}$ is a path in $G_m^0$.
The blocks visited by this path form a path in the block tree, so there are at most two visited leaf blocks, namely the two ends.
Every other visited block is a spine block.

Now partition all possible chords of $C$ according to their endpoints.
Let $e$ be a chord of $C$.
If one endpoint of $e$ is $v$, then the other endpoint must lie in a visited leaf block, since $v$ is only adjacent to leaf blocks.
In that case $e$ is counted in that leaf block.
If instead both endpoints of $e$ lie in the same visited block $Q$, then $e$ is counted in $Q$.
Finally, suppose the endpoints of $e$ lie in two distinct visited blocks.
Then $e$ cannot be incident to $v$, so it is an edge of $G_m^0$ joining two distinct blocks.
The only such edges are inter-block edges, and any inter-block edge whose endpoints both lie on $C$ is an edge of the path $C\backslash\{v\}$ itself, not a chord.
So this third case never occurs.
Thus every chord of $C$ is counted in exactly one visited leaf or spine block.

By Lemma~\ref{p1091:lem:endpoint-leaf}, at most $4$ chords are counted in each visited leaf block, for a total of at most $8$.
By Lemma~\ref{p1091:lem:spine-contrib}, no chord is counted in an internal spine block, and at most one chord is counted in each of the two end spine blocks.
Thus we conclude $\ch_{G_m}(C)\le 4+4+1+1=10$.
\end{proof}


Combining the preceding results, we conclude the main theorem of this section.

\begin{proof}[Proof of Theorem~\ref{p1091:thm:main}]
By construction there are $4m+6$ pentagon blocks, so
\[
|V(G_m)|=5(4m+6)+1=20m+31.
\]
Lemma~\ref{p1091:lem:k4free} gives the $K_4$-free property.
Propositions~\ref{p1091:prop:not-3-colorable} and Proposition~\ref{p1091:prop:degenerate} imply that $\chi(G_m)=4$ but that every subgraph omitting at least $1$ vertex is $3$-colorable.
Finally, Proposition~\ref{p1091:prop:chord-bound} gives the bound $\ch_{G_m}(C)\le10$
for every cycle $C$ in $G_m$.
\end{proof}

\section[Problem 990]{A counterexample to sparse Erd\H{o}s--Tur\'an}\label{sec:990}

This problem concerns the approximate uniformity of complex arguments for the roots of a polynomial, which is the subject of the famous Erd\H{o}s--Tur\'an theorem \cite{ET}. We provide a counterexample to a proposed strengthening in the case of a sparse polynomial with few non-zero coefficients.

\subsection{Introduction}

Let \(z_1,\dots,z_d\in \C^\ast\) be the zeros of a polynomial
\[
f(z)=\sum_{k=0}^d a_k z^k \qquad (a_0a_d\neq 0),
\]
counted with algebraic multiplicity, and let
\[
N_f([\alpha,\beta))
:=
\#\{1\le j\le d:\Arg(z_j)\in[\alpha,\beta)\}.
\]
Here \(\Arg\) denotes the principal argument in \([0,2\pi)\); below we use the parameters
\[
\nu(f):=\#\{0\le k\le d:a_k\neq 0\},
\qquad
M(f):=\frac{\sum_{k=0}^d |a_k|}{\sqrt{|a_0a_d|}}.
\]
A classical theorem of Erd\H{o}s and Tur\'an \cite{ET} shows that the arguments of the zeros are close to uniformly distributed, with discrepancy bounded by
\[
\left|N_f([\alpha,\beta))-\frac{\beta-\alpha}{2\pi}d\right|
\ll \sqrt{d\log M(f)}.
\]
There are several elegant proofs and variants of this theorem; see for instance Amoroso--Mignotte \cite{AM} and Soundararajan \cite{Sound}.

When \(f\) is sparse, it is natural to ask whether the degree \(d\) can be replaced by the number \(\nu(f)\) of nonzero coefficients. A result of Hayman \cite{Hayman} goes in this direction:
for every \(\nu(f)\)-nomial,
\[
\left|N_f([\alpha,\beta))-\frac{\beta-\alpha}{2\pi}d\right|
\le \nu(f)-1.
\]
See also Hrube\v{s} \cite{Hrubes}, who explains this bound through the number of positive roots of suitable rotated real parts. A bound of size
\[
O\!\left(\sqrt{\nu(f)\log M(f)}\right)
\]
would be a natural sparse strengthening of Erd\H{o}s--Tur\'an as speculated by Erd\H{o}s \cite{Erd64}. We demonstrate that such a strengthening need not hold. 

\begin{theorem}\label{p990:thm:main}
There is no absolute constant \(C>0\) with the following property:
for every polynomial
\[
f(z)=\sum_{k=0}^d a_k z^k \qquad (a_0a_d\neq 0)
\]
and every interval \(0\le \alpha<\beta\le 2\pi\),
\[
\left|N_f([\alpha,\beta))-\frac{\beta-\alpha}{2\pi}d\right|
\le C\sqrt{\nu(f)\log M(f)}.
\]
\end{theorem}


The construction is explicit. For each \(N\ge 1\), a polynomial \(f\) is produced with
\[
\nu(f)=N+2,\qquad M(f)<3,
\]
and with a positive real zero of multiplicity \(N+1\). Since a positive real zero has argument \(0\), the interval \([0,\pi/\deg f)\) then contains at least \(N+1\) zeros, whereas the uniform prediction is only \(1/2\).

\subsection{A Vandermonde identity}

The following elementary lemma will be used.

\begin{lemma}\label{p990:lem:vandermonde}
Let \(s\ge 2\), and let $\alpha_1<\alpha_2<\cdots<\alpha_s$
be distinct real numbers. Put
\[
P_j:=\prod_{\ell\neq j}(\alpha_j-\alpha_\ell),
\qquad
\Delta_j:=|P_j|.
\]
Then
\begin{equation}\label{p990:eq:vandermonde}
\sum_{j=1}^s \frac{(-1)^{s-j}\alpha_j^k}{\Delta_j}
=
\begin{cases}
0,& 0\le k\le s-2,\\[3pt]
1,& k=s-1.
\end{cases}
\end{equation}
\end{lemma}

\begin{proof}
Lagrange interpolation gives, for every polynomial \(q\) of degree at most \(s-1\),
\[
q(x)
=
\sum_{j=1}^s q(\alpha_j)\,
\frac{\prod_{\ell\neq j}(x-\alpha_\ell)}{P_j}.
\]
The numerator in the \(j\)-th term is monic of degree \(s-1\). Therefore the coefficient of \(x^{s-1}\) on the right-hand side is $\sum_{j=1}^s \frac{q(\alpha_j)}{P_j}$.
Choosing \(q(x)=x^k\) gives
\[
\sum_{j=1}^s \frac{\alpha_j^k}{P_j}
=
\begin{cases}
0,& 0\le k\le s-2,\\[3pt]
1,& k=s-1.
\end{cases}
\]
Finally, since \(\alpha_1<\cdots<\alpha_s\), the product \(P_j\) contains exactly \(s-j\) negative factors, so
\[
P_j=(-1)^{s-j}\Delta_j.
\]
Substituting this into the previous identity yields \eqref{p990:eq:vandermonde}.
\end{proof}

\subsection{The construction}

Fix an integer \(N\ge 1\) and an integer \(K\ge 2\). Set
\[
s:=N+2,
\qquad
d:=K^N,
\qquad
\varepsilon:=K^{-1}.
\]
Choose the exponents
\[
m_1=0,\qquad
m_{i+1}=K^{\,i-1}\ \ (1\le i\le N),
\qquad
m_s=K^N=d.
\]
Thus the support is
\[
0,\,1,\,K,\,K^2,\dots,K^{N-1},\,K^N.
\]
Next put $\lambda_j:=\frac{m_j}{d}$
so that
\[
\lambda_1=0,\qquad
\lambda_{i+1}=\varepsilon^{N+1-i}\ \ (1\le i\le N),
\qquad
\lambda_s=1.
\]
In particular,
\[
0=\lambda_1<\lambda_2<\cdots<\lambda_s=1.
\]

Define
\[
\Delta_j:=\prod_{\ell\neq j}|\lambda_j-\lambda_\ell|,
\qquad
A_1:=\sqrt{\frac{\Delta_s}{\Delta_1}},
\qquad
T:=\frac{1}{\sqrt{2}\,A_1},
\qquad
\tau:=2\log T.
\]
Finally, set
\begin{equation}\label{p990:eq:def-cj}
c_j:=\frac{(-1)^{s-j}e^{-\lambda_j\tau}}{\Delta_j}
\qquad (1\le j\le s),
\end{equation}
and define the lacunary polynomial
\[
f_{N,K}(z):=\sum_{j=1}^s c_j z^{m_j}.
\]
By construction \(f_{N,K}\) has exactly \(s=N+2\) nonzero coefficients, so $\nu(f_{N,K})=N+2$.

\begin{remark}
The endpoint coefficients are especially simple. Since \(\lambda_1=0\), \(\lambda_s=1\), and
\[
e^{-\tau}=T^{-2}=2A_1^2=\frac{2\Delta_s}{\Delta_1},
\]
one has $c_1=\frac{(-1)^{s-1}}{\Delta_1}$ and $c_s=\frac{2}{\Delta_1}$.
In particular, all coefficients are real and nonzero.
\end{remark}

The next lemma explains the choice of coefficients.

\begin{lemma}\label{p990:lem:multiple-root}
Let
\[
F_{N,K}(u):=\sum_{j=1}^s c_j e^{\lambda_j u}.
\]
Then \(F_{N,K}\) has a zero of order exactly \(s-1=N+1\) at \(u=\tau\). Consequently,
\[
x_0:=e^{\tau/d}>0
\]
is a zero of \(f_{N,K}\) of multiplicity \(N+1\).
\end{lemma}

\begin{proof}
Differentiating \(F_{N,K}\) gives
\[
F_{N,K}^{(k)}(\tau)
=
\sum_{j=1}^s c_j \lambda_j^k e^{\lambda_j\tau}
=
\sum_{j=1}^s \frac{(-1)^{s-j}\lambda_j^k}{\Delta_j}.
\]
Applying Lemma \ref{p990:lem:vandermonde} with \(\alpha_j=\lambda_j\) gives
\[
F_{N,K}^{(k)}(\tau)=0 \quad (0\le k\le s-2),
\qquad
F_{N,K}^{(s-1)}(\tau)=1.
\]
Hence \(u=\tau\) is a zero of exact order \(s-1\) of \(F_{N,K}\).
Next we reparametrize $F_{N,K}$ to be a polynomial. Let $x_0=e^{\tau/d}$.
For \(x>0\) one has
\[
F_{N,K}(d\log x)
=
\sum_{j=1}^s c_j e^{\lambda_j d\log x}
=
\sum_{j=1}^s c_j x^{m_j}
=
f_{N,K}(x).
\]
Fix a small disc \(U\subset \C\setminus\{0\}\) centered at \(x_0\), and a holomorphic branch of \(\log\) on \(U\). Then
\[
f_{N,K}(z)=F_{N,K}(d\log z),\quad\forall z\in U.
\]
Since \(z\mapsto d\log z\) is biholomorphic near \(x_0\), multiplicity is preserved. Therefore \(x_0\) is a zero of \(f_{N,K}\) of exact multiplicity \(s-1=N+1\).
\end{proof}

\begin{example}
The first nontrivial case is \(N=2\), for which the support is
\[
0,\,1,\,K,\,K^2.
\]
Then \(f_{2,K}\) has four nonzero terms and a triple positive root. For instance, when \(K=20\),
\[
f_{2,20}(z)\approx -8000+8647.7946547\,z-717.2170502\,z^{20}+16000\,z^{400},
\]
and numerically
\[
M(f_{2,20})\approx 2.9491,
\qquad
x_0=e^{\tau/400}\approx 0.9762209
\]
is a root of multiplicity \(3\).
\end{example}

\subsection{Bounded height}

The key point is that \(M(f_{N,K})\) stays bounded as \(K\to\infty\) for fixed \(N\).

\begin{proposition}\label{p990:prop:height}
For every fixed \(N\ge 1\),
\[
M(f_{N,K})\longrightarrow 2\sqrt{2}
\qquad\text{as }K\to\infty.
\]
\end{proposition}

\begin{proof}
Fix \(N\), and write \(\varepsilon=K^{-1}\to 0\). Here and below, \(O_N(\varepsilon)\)
means a quantity whose absolute value is at most \(C_N\varepsilon\) for some constant
\(C_N\) depending only on this fixed \(N\). Define
\[
A_j:=\frac{\sqrt{\Delta_1\Delta_s}}{\Delta_j}
\qquad (1\le j\le s).
\]
Using \eqref{p990:eq:def-cj} and \(T=e^{\tau/2}\) gives the exact formula
\begin{equation}\label{p990:eq:M-exact}
M(f_{N,K})=\sum_{j=1}^s A_j\,T^{\,1-2\lambda_j}.
\end{equation}

It is convenient to write
\[
x_i:=\lambda_{i+1}=\varepsilon^{N+1-i}
\qquad (1\le i\le N).
\]
Then
\[
\Delta_1=\prod_{i=1}^N x_i=\varepsilon^{N(N+1)/2},
\qquad
\Delta_s=\prod_{i=1}^N(1-x_i)=1+O_N(\varepsilon).
\]
For \(1\le i\le N\), the factors in \(\Delta_{i+1}\) are
\[
|x_i-0|=x_i,\qquad |1-x_i|=1-x_i,
\]
\[
|x_i-x_r|=x_i(1-\varepsilon^{\,i-r}) \quad (r<i),
\qquad
|x_i-x_r|=x_r(1-\varepsilon^{\,r-i}) \quad (r>i).
\]
Hence
\[
\Delta_{i+1}
=
x_i(1-x_i)
\prod_{r=1}^{i-1}x_i(1-\varepsilon^{\,i-r})
\prod_{r=i+1}^{N}x_r(1-\varepsilon^{\,r-i})
=
x_i^{\,i}\prod_{r=i+1}^{N}x_r\,\bigl(1+O_N(\varepsilon)\bigr).
\]
Dividing by \(\Delta_1\) gives
\begin{equation}\label{p990:eq:ratio-deltas}
\frac{\Delta_{i+1}}{\Delta_1}
=
\varepsilon^{-i(i-1)/2}\bigl(1+O_N(\varepsilon)\bigr)
\qquad (1\le i\le N).
\end{equation}
Therefore
\begin{equation}\label{p990:eq:A1}
A_1
=
\sqrt{\frac{\Delta_s}{\Delta_1}}
=
\varepsilon^{-N(N+1)/4}\bigl(1+O_N(\varepsilon)\bigr),
\end{equation}
and, with
\[
R_i:=\frac{A_{i+1}}{A_1}=\frac{\Delta_1}{\Delta_{i+1}},
\]
one has
\begin{equation}\label{p990:eq:Ri}
R_i
=
\varepsilon^{i(i-1)/2}\bigl(1+O_N(\varepsilon)\bigr)
\qquad (1\le i\le N).
\end{equation}
In particular,
\[
R_1\to 1,
\qquad
R_i\to 0 \quad (2\le i\le N).
\]

There are also the exact identities
\begin{equation}\label{p990:eq:endpoint-id}
A_sA_1=1,
\qquad
A_1T=\frac{1}{\sqrt{2}}.
\end{equation}
Next, since \(T=(\sqrt{2}A_1)^{-1}\), \eqref{p990:eq:A1} implies
\[
|\log T|=O_N(|\log\varepsilon|).
\]
Because \(x_i\le \varepsilon\), it follows that
\[
|2x_i\log T|=O_N(\varepsilon|\log\varepsilon|)\to 0,
\]
and so
\begin{equation}\label{p990:eq:small-power}
T^{-2x_i}=1+o(1)
\qquad (1\le i\le N),
\end{equation}
uniformly for fixed \(N\).

Now split \eqref{p990:eq:M-exact} into the first term, the middle terms, and the last term:
\[
M(f_{N,K})
=
A_1T+\sum_{i=1}^N A_{i+1}T^{1-2x_i}+A_sT^{-1}.
\]
Using \eqref{p990:eq:endpoint-id}, \eqref{p990:eq:Ri}, and \eqref{p990:eq:small-power}, this gives
\[
M(f_{N,K})
=
\frac{1}{\sqrt{2}}
+
\frac{1}{\sqrt{2}}\sum_{i=1}^N R_i\,T^{-2x_i}
+
\sqrt{2}.
\]
The term \(i=1\) tends to \(1/\sqrt{2}\), while every term \(i\ge 2\) tends to \(0\). Therefore
\[
M(f_{N,K})\longrightarrow
\frac{1}{\sqrt{2}}+\frac{1}{\sqrt{2}}+\sqrt{2}
=
2\sqrt{2}.
\]
This proves the proposition.
\end{proof}

\subsection{Proof of the main theorem}

The argument can now be completed.

\begin{proof}[Proof of Theorem \ref{p990:thm:main}]
Fix \(N\ge 1\). Only fixed-\(N\) asymptotics are needed here, since \(K\) will be
chosen after \(N\) has been fixed. By Proposition \ref{p990:prop:height}, there exists
\(K=K(N)\) so large that
\[
M(f_{N,K})<3.
\]
Set
\[
f_N:=f_{N,K(N)},
\qquad
d_N:=K(N)^N.
\]
By construction,
\[
\nu(f_N)=N+2,
\qquad
M(f_N)<3.
\]
By Lemma \ref{p990:lem:multiple-root}, the polynomial \(f_N\) has a positive real zero \(x_{0,N}\) of multiplicity \(N+1\). Every copy of this zero has principal argument \(0\).

Consider the interval
\[
I_N:=\left[0,\frac{\pi}{d_N}\right).
\]
Its expected number of zeros under uniform angular distribution is
\[
\frac{|I_N|}{2\pi}d_N
=
\frac{\pi/d_N}{2\pi}d_N
=
\frac{1}{2}.
\]
However \(I_N\) contains all \(N+1\) copies of the positive zero \(x_{0,N}\), so $N_{f_N}(I_N)\ge N+1$ and thus
\[
\left|N_{f_N}(I_N)-\frac{|I_N|}{2\pi}d_N\right|
=
N_{f_N}(I_N)-\frac12
\ge N+\frac12.
\]

Suppose, for contradiction, that the theorem were false. Then there would exist an absolute constant \(C>0\) such that
\[
\left|N_f(I)-\frac{|I|}{2\pi}\deg f\right|
\le C\sqrt{\nu(f)\log M(f)}
\]
for every polynomial \(f\) and every interval \(I\). Applying this to \(f_N\) and \(I_N\) gives
\[
N+\frac12
\le
C\sqrt{(N+2)\log M(f_N)}
\le
C\sqrt{(N+2)\log 3}.
\]
This is impossible for large \(N\), hence no such absolute constant \(C\) exists.
\end{proof}

\section[On primes of the form n-a k squared]{On primes of the form \texorpdfstring{$n-a k^2$}{n-a k squared}}\label{sec:nt}

Fix an integer $a\ge 1$. Let $P_a(n)$ denote the property that
\[
n-a k^2 \text{ is prime for every integer }k\ge 1\text{ with }(k,n)=1\text{ and }a k^2<n.
\]
We prove that for each fixed $a$, only finitely many integers satisfy $P_a(n)$. 
The case $a=1$ is Erd\H{o}s's Problem~1141 \cite{Erd99,BloWeb}.
The proof is a short deduction from Pollack's theorem \cite[Theorem~1.3]{Pollack17}.

\begin{theorem}\label{thm:nt}
Fix $a\ge 1$. There are only finitely many integers $n$ such that $P_a(n)$ holds.
\end{theorem}

\begin{remark}
Due to the use of Siegel's theorem in Pollack's argument, Theorem~\ref{thm:nt} is ineffective. In the original case $a=1$, computational evidence suggests that the maximal such $n$ is $1722$.
\end{remark}

The key input is the following result of Pollack \cite[Theorem~1.3]{Pollack17}.

\begin{theorem}\label{thm:input}
Let $A\ge 1$ and $\eps>0$. Then there exists $M_0= M_0(A,\eps)\ge 1$ such that if $m\ge M_0$ and $\chi$ is a quadratic character modulo $m$, there are at least $(\log m)^{A}$ primes $p\le m^{1/4 + \eps}$ with $\chi(p) = 1$.
\end{theorem}

\begin{proof}[Proof of Theorem~\ref{thm:nt}]
Fix $a\ge 1$, and suppose $P_a(n)$ holds for some sufficiently large $n$.
Write
\[
a n=u^2d,
\qquad
d\text{ squarefree}.
\]
We split into two cases.

\smallskip
\noindent\textit{Case 1: $d>1$ (equivalently, $an$ is not a square).}
Let $\chi$ be the nontrivial quadratic character attached to $\mathbb Q(\sqrt d)$, viewed modulo $4an$.
For every odd prime $p\nmid an$,
\[
\chi(p)=1
\iff d\text{ is a square mod }p 
\iff a x^2\equiv n\pmod p\text{ is solvable}.
\]
With $\omega(m)$ the number of distinct prime factors, it is well known that $\omega(m)\leq o(\log m)$.
Applying Theorem~\ref{thm:input} with $\eps=\tfrac18$, $A=1$, and $m=4an$, we find that for all sufficiently large $n$, there exists an odd prime $p\nmid an$ with
\[
p\le (4an)^{3/8}\ll_a n^{3/8}
\]
and $\chi(p)=1$.
Hence the congruence $a x^2\equiv n\pmod p$ has exactly two roots $r_1,r_2$ modulo $p$.

Define
\[
S:=\{k\ge 1:k<\sqrt{n/a},\ k\equiv r_1\text{ or }r_2\pmod p,\ (k,n)=1\}.
\]
If $k\in S$, then $p\mid n-a k^2$. Since $a k^2<n$ and $P_a(n)$ holds, the number $n-a k^2$ is prime, so necessarily
\[
n-a k^2=p.
\]
But this equation has at most one positive integer solution $k$. Therefore $|S|\leq 1$.

Now we estimate $|S|$ from below, using inclusion-exclusion to handle the relative primality constraint.
Set
\[
M:=\frac{\sqrt{n/a}}{p}.
\]
For each $i\in\{1,2\}$, the number of integers $t\ge0$ with $r_i+t p<\sqrt{n/a}$ is $M+O(1)$.
By M\"obius inversion over $\operatorname{rad}(n)$,
\[
\#S
=\sum_{i=1}^2\ \sum_{m\mid\operatorname{rad}(n)}\mu(m)\,\#\{t\ge0:r_i+t p<\sqrt{n/a},\ r_i+t p\equiv0\pmod m\}.
\]
Since $p\nmid n$, each congruence in $t$ gives one residue class modulo $m$, hence
\[
\#S
=
2M\frac{\varphi(n)}{n}+O(2^{\omega(n)})
=
\frac{2\sqrt{n/a}}{p}\frac{\varphi(n)}{n}
+O(2^{\omega(n)}).
\]
Using $\varphi(n)/n\gg 1/\log\log n$, $2^{\omega(n)}=n^{o(1)}$, and $p\ll_a n^{3/8}$, we get
\[
\#S\gg_a \frac{n^{1/8}}{\log\log n}>1
\]
for all sufficiently large $n$, contradiction.

\smallskip
\noindent\textit{Case 2: $d=1$ (equivalently, $an$ is a square).}
Then for every odd prime $p\nmid an$, the congruence $a x^2\equiv n\pmod p$ is automatically solvable.
Let $p$ be the least odd prime with $p\nmid an$.
A standard estimate via e.g.\ the prime number theorem gives
\[
p\ll \log(an)\ll_a \log n.
\]
Choose two roots $r_1,r_2\pmod p$ of $a x^2\equiv n\pmod p$, and define $S$ exactly as above.
The same argument yields
\[
\#S
=
\frac{2\sqrt{n/a}}{p}\frac{\varphi(n)}{n}
+O(2^{\omega(n)})
\gg_a
\frac{\sqrt n}{\log n\,\log\log n}>1
\]
for all sufficiently large $n$, again contradicting $\#S\le 1$ and completing the proof.
\end{proof}

\bibliographystyle{amsplain0}
\bibliography{main}

\end{document}